\documentclass[12pt,dvipdfmx]{amsart}

\usepackage{amsmath}
\usepackage{amssymb}
\usepackage{amsfonts}
\usepackage{amsthm}
\usepackage{amscd}
\usepackage{latexsym}
\usepackage[dvipdfmx]{graphicx}
\usepackage{color}
\usepackage{bxwareki}
\usepackage[all]{xy}

\newcommand{\wt}{\widetilde}
\newcommand{\wh}{\widehat}
\newcommand{\supp}{{\operatorname{supp\,}}}

\newcommand{\ispa}[1]{\langle \,#1 \,\rangle } 
 
\newcommand{\comp}{{\scriptstyle \circ}} 
 
\newcommand{\spec}{\operatorname{Spec}\nolimits}
 
\newcommand{\ol}{\overline}

\newcommand{\mf}{\mathfrak}
\newcommand{\mb}{\mathbb}
\newcommand{\mbf}{\mathbf}
\newcommand{\mr}{\mathrm}
\newcommand{\mc}{\mathcal}
\newcommand{\msc}[1]{\mbox{\textsc{#1}}}

\newcommand{\ve}{\varepsilon}
\newcommand{\dsp}{\displaystyle}

\newtheorem{thm}{{\bf Theorem}}[section]
\newtheorem*{thm*}{{\bf Theorem}}
\newtheorem{cor}[thm]{{\bf Corollary}}
\newtheorem{lem}[thm]{{\bf Lemma}}
\newtheorem{prop}[thm]{{\bf Proposition}}
\theoremstyle{definition}

\newtheorem*{rem}{{\bf Remark}}

\renewenvironment{proof}%
{\def\psymbol{{\bf Proof: }\hspace{2pt}}
\par\noindent\psymbol}%
{\def\qed{$\blacksquare$}
\hfill\qed\par\bigskip}

\begin{document}
\title[Horizontal Laplacian with integrable horizontal distribution]
{The horizontal Laplacian of a Riemannian submersion \\ 
with totally geodesic fibers \\
and an integrable horizontal distribution}
\author{Tatsuya Tate}
\address{Mathematical Institute, Graduate School of Sciences, Tohoku University, 
Aoba, Sendai 980-8578, Japan. }
\email{tatsuya.tate.c6@tohoku.ac.jp}
\thanks{The author is partially supported by JSPS KAKENHI Grant Number 23K25769, 24K06703.}

\begin{abstract}
The purpose of this note is to study spectral properties of the horizontal Laplacian 
of a Riemannian submersion with totally geodesic fibers and an integrable horizontal distribution. 
We show that the horizontal Laplacian is unitarily equivalent to 
a twisted Laplacian acting on the space of sections of a certain infinite-rank 
flat vector bundle over the base manifold of the Riemannian submersion. 
We give an application of this interpretation to the asymptotic behavior 
of the scaled first nonzero eigenvalue of the canonical variations introduced 
by B\'{e}rard-Bergery and Bourguignon. 
Our approach enables us to compare the horizontal Laplacian 
with the usual Laplacian on a Riemannian covering over the base manifold, 
and, when the holonomy group is infinite and amenable, we prove a coincidence 
of the essential spectrum, which strengthen, in our special setup, 
a result due to Kordyukov in the context of geometric analysis on foliated manifolds. 
\end{abstract}

\renewcommand{\thefootnote}{\fnsymbol{footnote}}
\renewcommand{\labelenumi}{{\rm (\arabic{enumi})}}
\renewcommand{\labelenumii}{{\rm (\alph{enumii})}}

\maketitle

\section{Introduction}\label{INT}
In 1982, B\'{e}rard-Bergery and Bourguignon \cite{BBB} introduced the 
{\it horizontal Laplacian} on the total space of a Riemannian submersion 
as a difference between the usual Laplacian and the {\it vertical Laplacian}, 
and they used it to study spectral properties of the Laplacian. 
One of their motivation was to systematically understand Urakawa's example 
\cite{Ur} of the metric on the $3$-sphere whose Laplacian has the first nonzero eigenvalue with multplicity 7. 
This metric, which is often called the {\it Berger metric}, 
is defined by a scaling along the fibers of the Hopf fibration. 
Taking this into account, they introduced the so-called {\it canonical variations} 
of the given Riemannian submersion. 
The horizontal Laplacian is well-behaved when the fibers of the submersion are totally geodesic, 
and it becomes a useful tool to understand the behavior of the first nonzero eigenvalues of the canonical variations. 
Since then, the study of spectral properties of Laplacians on Riemannian submersions 
with totally geodesic fibers is one of major subjects in Spectral Geometry. 
For example, Besson \cite{Be} obtained a characterization of the Riemannian product 
in terms of the equality condition of an inequality on the heat kernel 
Besson-Bordoni \cite{BB} studied the eigenvalues of the Laplacian 
using a representation theory of the isometry group of the fiber of the Riemannian submersion.  
Recently, Narita \cite{N} gives a sufficient condition in terms of the Ricci curvatures 
for the divergence of the scaled first nonzero eigenvalues of the canonical variations with the parameter goes to infinity. 
The list of references in \cite{N} will be helpful to find other works closely related to this vast topic. 
\par
When the horizontal distribution of the Riemannian submersion is integrable, 
the Riemannian structure is very close to the Riemannian product. 
However, as is shown in the example of the irrationally `tilted' Riemannian metric on the torus, 
which is mentioned in \cite{BBB}, the spectral properties of the corresponding 
horizontal Laplacian is still far from trivial. 
Our purpose in this paper is to give several spectral properties of the horizontal Laplacian 
under the strong assumption that, besides that the fibers are totally geodesic, the horizontal distribution is integrable. 
The properties we obtain here comes from an interpretation of the horizontal Laplacian 
as a twisted Laplacian on a certain flat vector bundle. 
To explain our setup and results, let us prepare some notation. 
Let $(M,g)$, $(B,j)$ be compact connected 
Riemannian manifolds of dimension $n$ and $p$, respectively, 
and let $\pi \colon M \to B$ be a Riemannian submersion. 
For each $x \in M$, set $\mc{V}_{x}=\ker d\pi_{x}$ and write $\mc{H}_{x}$ for the 
orthogonal complement of $\mc{V}_{x}^{\perp}$ in the tangent space $T_{x}M$. 
The distributions $\mc{V}$ and $\mc{H}$ are called the {\it vertical} and 
the {\it horizontal distributions}, respectively. 
The Laplacian on $(M,g)$, denoted by $\Delta_{M}$, has a decomposition, 
\[
\Delta_{M}=\Delta_{H} + \Delta_{V}, 
\]
where $\Delta_{V}$ is the vertical Laplacian defined by 
\[
(\Delta_{V}f)(x)=(\Delta_{F_{\pi(x)}} f|_{\pi^{-1}(\pi(x))}) (x), 
\]
with the Laplacian $\Delta_{F_{b}}$ on the fiber $F_{b}=\pi^{-1}(b)$ over a point $b \in B$ 
equipped with the metric induced by $g$, and, by definition, $\Delta_{H}=\Delta_{M} - \Delta_{V}$. 
The differential operator $\Delta_{H}$ is called the horizontal Laplacian. 
It was observed by B\'{e}rard-Bergery-Bourguignon \cite{BBB} that $\Delta_{H}$ and $\Delta_{V}$ 
are commutative on $C^{\infty}(M)$ if the Riemannian submersion has totally geodesic fibers. 
Suppose now that, for any $b \in B$, the fiber $F_{b}$ is a totally geodesic submanifold of $M$.  
Under this assumption, Hermann \cite{Her1} shows that every two fibers are isometric to each other, 
and $\pi \colon M \to B$ becomes a fiber bundle with structure group the isomorphism group, $\mr{Isom}(F)$, 
of a typical fiber $F=F_{b_{o}}$ over a fixed base-point $b_{o} \in B$. 
Since $M$ is compact and $\Delta_{M}$ is elliptic, there exists an orthonormal basis 
of $L^{2}(M)$ consisting of smooth joint eigenfunctions of $\Delta_{M}$, $\Delta_{H}$ and $\Delta_{V}$. 
Thus the operators $\Delta_{M}$, $\Delta_{V}$, $\Delta_{H}$, 
with the space of smooth functions on $M$ as their domain, are essentially self-adjoint. 
The closure of these operators are denoted by the same notation. 
In general, let $\spec(S)$ denote the spectrum of a self-adjoint operator $S$ on a Hilbert space 
and let $\spec_{\mr{ess}}(S)$ denote the essential spectrum of $S$. 
As a set, we have $\spec(\Delta_{V})=\spec(\Delta_{F})$, where $\Delta_{F}$ denotes the 
Laplacian of the metric on $F$ induced by $g$, and $\spec(\Delta_{V})$ consists of 
eigenvalues, but each eigenvalue of $\Delta_{V}$ could have an infinite multiplicity. 
More subtle is the spectrum of the horizontal Laplacian $\Delta_{H}$, 
which is our main object under the assumption that the horizontal distribution $\mc{H}$ is integrable. 
\par
Now, suppose that each of the fiber of $\pi \colon M \to B$ is connected and totally geodesic, 
and furthermore, the horizontal distribution $\mc{H}$ is integrable. 
Under these assumptions, we have the holonomy representation (see Lemma \ref{Hol1}) 
\begin{equation}\label{HO}
\mf{h} \colon \pi_{1}(B,b_{o}) \to \mr{Isom}(F)
\end{equation}
of the fundamental group $\pi_{1}(B,b_{o})$ with the base-point $b_{o} \in B$, 
and its image is the holonomy group $\mr{Hol}(b_{o})$ at $b_{o}$ 
(see Section \ref{FL} for the definition). 
Let $\wh{p} \colon \wh{B} \to B$ be the Riemannian universal covering over $B$. 
Later, we impose assumptions on $\mr{Hol}(b_{o})$ instead of $\pi_{1}(B,b_{o})$. 
Thus, we prepare a slightly different Riemannian covering. 
Define $B_{H}=\wh{B}/\mr{ker}(\mf{h})$ with a natural projection $\varpi_{H} \colon B_{H} \to B$. 
Equip the pull-back metric $\varpi_{H}^{*}j$ on $B_{H}$ so that $\varpi_{H} \colon B_{H} \to B$ is a 
regular Riemannian covering. The covering transformation group of $\varpi_{H}$ is $\mr{Hol}(b_{o})$. 
Since $\mr{Hol}(b_{o})$ is a subgroup of $\mr{Isom}(F)$, we have a unitary representation 
\begin{equation}\label{HUR}
\rho \colon \mr{Hol}(b_{o}) \to U(L^{2}(F)), \quad 
(\rho(h)f)(p)=f(h^{-1}p) \ \ (h \in \mr{Hol}(b_{o}), f \in L^{2}(F),\ p \in F), 
\end{equation}
and the corresponding flat vector bundle 
\begin{equation}\label{FVB1}
\pi_{\rho} \colon \mc{E}_{\rho}=B_{H} \times_{\rho} L^{2}(F) \to B, 
\end{equation}
whose fibers are isometrically isomorphic to $L^{2}(F)$. 
As in \cite{Su2}, all differential operators on $B$ are extended to differential operators acting 
on the space of smooth sections $C^{\infty}(\mc{E}_{\rho})$ of $\mc{E}_{\rho}$. 
Let $\Delta_{\rho}$ denote the extension of the Laplacian $\Delta_{B}$ on $B$ to $C^{\infty}(\mc{E}_{\rho})$, 
which is often called the {\it Laplacian twisted by the representation} $\rho$. 
It is well-known that the operator $\Delta_{\rho}$ is essentially self-adjoint on $C^{\infty}(\mc{E}_{\rho})$ 
(which is mentioned in \cite{Su1}, and its proof using the heat kernel can be found in \cite{Su2}). 
For $s_{1}, s_{2} \in C^{\infty}(\mc{E}_{\rho})$, we define the inner product $\ispa{s_{1},s_{2}}_{L^{2}(\mc{E}_{\rho})}$ by 
\[
\ispa{s_{1},s_{2}}_{L^{2}(\mc{E}_{\rho})} =
\int_{B} \ispa{s_{1}(b), s_{2}(b)}_{\pi_{\rho}^{-1}(b)} \,d\mu_{B}(b), 
\]
where $d\mu_{B}$ is the Riemannian volume measure of $(B,j)$ 
and $\ispa{\cdot, \cdot}_{\pi_{\rho}^{-1}(b)}$ is the Hilbert space inner product on the fiber $\pi_{\rho}^{-1}(b)$. 
Let $L^{2}(\mc{E}_{\rho})$ denote the Hilbert space obtained 
by completing $C^{\infty}(\mc{E}_{\rho})$ with respect to the norm defined by the above inner product. 
\par
When $M$ is the Riemannian product $B \times F$, Fubini's theorem shows 
that $L^{2}(M)$ is isomorphic to the space of $L^{2}$-sections of the trivial vector bundle $B \times L^{2}(F)$ 
over $B$ and the horizontal Laplacian on $B \times F$ is unitarily equivalent to the Laplacian acting 
on the space of sections of $B \times L^{2}(F) \to B$. 
The following theorem generalizes this observation. 
%
%
\begin{thm}\label{MN1}
Under the assumptions mentioned above, 
$\Delta_{\rho}$ is unitarily equivalent to $\Delta_{H}$. 
\end{thm}
It is natural to expect the assertion in Theorem $\ref{MN1}$ because, as is proved in Section $\ref{FL}$, 
$M$ is isometrically isomorphic, as a fiber bundle, to $B_{H} \times_{\mr{Hol}(b_{o})} F$. 
However, this rather natural Theorem $\ref{MN1}$ has some applications. 
For $a \geq 0$, we set 
\begin{equation}\label{ESP1}
C^{\infty}(M,\mu)=\{f \in C^{\infty}(M) \mid \Delta_{V}f=\mu f\}, \quad 
E(\mu,a)=\{f \in C^{\infty}(M, \mu) \mid \Delta_{H}f=af\}. 
\end{equation}
We denote $E(\mu)$ the eigenspace of $\Delta_{F}$ with the eigenvalue $\mu$. 
Then the unitary operator intertwining $\Delta_{H}$ with $\Delta_{\rho}$ 
maps $C^{\infty}(M,\mu)$ onto the space of smooth sections of the finite-dimensional flat vector bundle 
\begin{equation}\label{FVB2}
\mc{E}_{\rho_{\mu}}=\wh{B} \times_{\rho_{\mu}} E(\mu), 
\end{equation}
where $\rho_{\mu}$ denotes the restriction of $\rho$ on $E(\mu)$, and it intertwines 
the restriction of $\Delta_{H}$ on $C^{\infty}(M,\mu)$ with the twisted Laplacian $\Delta_{\rho_{\mu}}$. 
Therefore, the following is a direct consequence of Theorem $\ref{MN1}$. 
%
%
\begin{cor}\label{TEN}
Suppose that the assumptions mentioned above hold true. 
For each $\mu \in \spec(\Delta_{F})$, there are countably infinitely many $a \geq 0$ such that 
$E(\mu,a) \neq \{0\}$. Number such $a \geq 0$ as 
\[
0 \leq a_{0}(\mu) \leq a_{1}(\mu) \leq a_{2}(\mu) \leq \cdots \leq a_{k}(\mu) \leq \cdots,  
\]
where, a fixed number $a$ is multiply counted with $\dim E(\mu,a)$ times. 
Then, for each $\mu \in \spec(\Delta_{F})$, we have 
\[
a_{k}(\mu) \sim 4\pi C^{2/\dim(B)} k^{2/\dim B} \quad (k \to \infty), \quad 
C=\frac{\Gamma(\frac{p}{2}+1)}{m(\mu)\mr{vol}(B)}, 
\]
where $m(\mu)$ is the multiplicity of $\mu \in \spec(\Delta_{F})$. 
\end{cor}
The next application of Theorem $\ref{MN1}$ is about asymptotic behavior of 
the scaled first nonzero eigenvalue of the canonical variations $\{g_{t}\}_{t>0}$ of 
the metric $g$ on $M$, which is defined by (\cite{BBB}) 
\begin{enumerate}
\item[(i)] $g_{t}|_{\mc{V}_{x} \times \mc{H}_{x}}=0$; 
\item[(ii)] $g_{t}|_{\mc{V}_{x} \times \mc{V}_{x}}=t^{2}g|_{\mc{V}_{x} \times \mc{V}_{x}}$; 
\item[(iii)] $g_{t}|_{\mc{H}_{x} \times \mc{H}_{x}}=g|_{\mc{H}_{x} \times \mc{H}_{x}}$. 
\end{enumerate}
In general, for a Riemannian manifold $(N,h)$, the scaled first nonzero eigenvalue is defined by 
\begin{equation}\label{Ber}
\Lambda_{1}(N,h)=\lambda_{1}(N,h) \mr{vol}(N,h)^{2/\dim N}, 
\end{equation}
where $\lambda_{1}(N,h)$ is the first nonzero eigenvalue of the Laplacian 
associated with the Riemannian metric $h$. 
The quantity $\Lambda_{1}(N,h)$ is invariant under the scaling $h \mapsto ch$ 
by a positive constant $c$ and is used to understand the 
behavior of $\lambda_{1}(N,h)$ as a function of the metric $h$ on $N$. 
When $M$ is the Riemannian product, $B \times F$, we have $\dsp \lim_{t \to \infty} \Lambda_{1}(M,g_{t})=0$. 
However, for the `irrational torus' $T(a)$ with a quadratic irrational number $a$, 
there exists a constant $c>0$ such that $c \leq \Lambda_{1}(M,g_{t}) \leq 8\pi$ for $t \geq 1$. 
Due to the totally geodesic and the integrability assumptions, 
our $(M,g)$ is close to $B \times F$ with the Riemannian product metric. 
But, we have the following. 
\begin{thm}\label{BBC}
Suppose that $\mr{Hol}(b_{o})$ satisfies Kazhdan's property $\mr{(T)}$ and 
the action of $\mr{Hol}(b_{o})$ on $F$ is ergodic. Then we have 
\[
\lim_{t \to \infty} \Lambda_{1}(M,g_{t}) =+\infty. 
\]
\end{thm}
As is observed by Sullivan \cite{Sul}, there is an arithmetic lattice $\Gamma_{n}$ in $O(n,2)$ with $n \geq 3$ 
satisfying Kazhdan's property $\mr{(T)}$ and has an injective homomorphism $f \colon \Gamma_{n} \to SO(n+2)$ 
with a dense image. Hence, this lattice acts on the sphere ergodically. Since the arithmetic lattices 
are finitely presented (see, for example, \cite{DW}), 
there is a compact manifold $B$ such that $\pi_{1}(B)$ is isomorphic to $\Gamma_{n}$. 
The manifold $M=\wh{B} \times_{f} S^{n+1}$ with the Riemannian metric 
induced by the product metric on $\wh{B} \times S^{n+1}$ satisfies all of our assumptions 
(see Proposition $\ref{FPL}$), and hence Theorem $\ref{BBC}$ can be applied. 
\par
The horizontal Laplacian in our setup is an example of tangentially elliptic 
operators in the context of geometric analysis of foliated manifolds. 
In particular, Kordyukov \cite{Kor0}, \cite{Kor} obtained a `coincidence' theorem which asserts that, 
under an amenability assumption of the $C^{*}$-algebras defined by the holonomy groupoid and 
the assumption that each leaf of the foliation $\mc{H}$ is dense in $M$, 
the spectrum of a tangentially elliptic operator and that of its restriction to each leaf coincide as a set. 
See also \cite{An} for a further generalization. Related to Kordyukov's work, we have the following. 
\begin{thm}\label{BOT}
Under the assumptions mentioned above, we have 
$\spec(\Delta_{B_{H}}) \subset \spec(\Delta_{H})$. 
If the holonomy group $\mr{Hol}(b_{o})$ of $\pi \colon M \to B$ is infinite, 
we have 
\[
\spec(\Delta_{B_{H}}) =\spec_{\mr{ess}}(\Delta_{B_{H}}) \subset \spec_{\mr{ess}} (\Delta_{H}). 
\]
Next, suppose that $\mr{Hol}(b_{o})$ is amenable. 
Then we have $\spec(\Delta_{B_{H}}) = \spec(\Delta_{H})$. 
Furthermore, if $\mr{Hol}(b_{o})$ is infinite and amenable, 
then we have 
\[
\spec_{\mr{ess}}(\Delta_{B_{H}})=\spec_{\mr{ess}}(\Delta_{H})=\spec(\Delta_{H}). 
\]
\end{thm}
The fact that $\Delta_{B_{H}}$ does not have eigenvalues with finite multiplicities if 
$\mr{Hol}(b_{o})$ is infinite follows from a theorem due to Polymerakis \cite{Pol}. 
Note that, even if $\mr{Hol}(b_{o})$ is amenable and infinite, $\Delta_{H}$, and hence $\Delta_{\rho}$, 
can have eigenvalues with finite multiplicities. According to Theorem $\ref{BOT}$, 
such eigenvalues are not isolated in the spectrum. 
\par
Let $L(p)$ denote the leaf of the integrable horizontal distribution through $p \in F$. 
Consider the metric on $L(p)$ induced by the metric $g$ on $M$, 
and denote $\Delta_{L(p)}$ the Laplacian on $L(p)$. Since $L(p)$ is complete 
(see, for example, \cite{ON2}, \cite{KN}), $L(p)$ is a Riemannian covering over $B$, and $\Delta_{p}$ 
is essentially self-adjoint on $C^{\infty}_{0}(L(p))$. Furthermore, if the stabilizer of $p$ in $\mr{Hol}(b_{o})$ 
is trivial, then the Riemannian covering $\pi_{L(p)} \colon L(p) \to B$ is isomorphic to 
$\varpi_{H} \colon B_{H} \to B$. Hence we have the following. 
\begin{cor}\label{COR2}
In addition to the assumptions mentioned above, 
we further assume that $\mr{Hol}(b_{o})$ is amenable. 
If the stabilizer of $p \in F$ in $\mr{Hol}(b_{o})$ is trivial, then 
we have $\spec(\Delta_{L(p)})=\spec(\Delta_{H})$. 
If furthermore $\mr{Hol}(b_{o})$ is infinite, then we have 
$\spec_{\mr{ess}}(\Delta_{L(p)})=\spec_{\mr{ess}}(\Delta_{H})$. 
\end{cor}
The organization is as follows. 
In Section \ref{FL}, we give some notes on the structure of the manifold $M$ 
under the assumptions mentioned above and give proofs of Theorem $\ref{MN1}$ and Corollary $\ref{TEN}$. 
A proof of Theorem $\ref{BBC}$ is given in Section $\ref{CV}$ which uses Sunada's inequality \cite{Su1} on 
the bottom of the spectrum of the twisted Laplacian and Kazhdan's constant of the corresponding 
unitary representation of $\pi_{1}(B)$. 
Theorem $\ref{BOT}$ is proved in Section $\ref{TM2}$. 
To prove the assertion on the essential spectrum in Theorem $\ref{BOT}$, 
we construct a Weyl sequence for $\Delta_{\rho}$ from that for $\Delta_{B_{H}}$. 
Our method of construction is to compare $\Delta_{B_{H}}$ and $\Delta_{\rho}$ 
directly by using a periodification of functions in $C_{0}^{\infty}(\Delta_{B_{H}})$ and 
the fact that the action of $\mr{Hol}(b_{o})$ on $F$ is isometric. 
\par
\vspace{15pt}
\noindent{\bf Acknowledgments:} 
\par
\vspace{5pt}
The author would like to thank Professors Motoko Kotani, Shin Nayatani, and Kazumasa Narita for their stimulating discussions. 
In particular, Theorem $\ref{BBC}$ has been found by the author as an answer to Professor Nayatani's question 
to the summary of the very earlier version of the manuscript. 
\section{A flat vector bundle and a twisted Laplacian} \label{FL}
In this section we give proofs of Theorem $\ref{MN1}$ and Corollary $\ref{TEN}$. 
We first show that Riemannian submersions with totally geodesic fibers 
and integrable horizontal distributions can be written as fiber products.  
The unitary intertwiner between the horizontal Laplacian and the twisted Laplacian 
is naturally defined by using this interpretation. 
\subsection{Fiber products}\label{Fib}
Let us begin by recalling the definition of the holonomy group of the Riemannian submersion. 
Let $(M,g)$ and $(B,j)$ be complete connected Riemannian manifolds and let $\pi \colon M \to B$ 
be a Riemannian submersion,  
that is, $\pi$ is surjective and, for each $x \in M$, the restriction of $d\pi_{x}$ to the horizontal subspace $\mc{H}_{x}$ 
is an isometric isomorphism onto $T_{\pi(x)}B$. 
Let $c \colon [0,1] \to B$ be a (piecewise smooth) path. 
For each $p \in F_{c(0)}$ there exists a unique path $\ol{c_{p}} \colon [0,1] \to M$ 
such that $\ol{c_{p}}'(t) \in \mc{H}_{\ol{c_{p}}(t)}$ for each $t \in [0,1]$, $\pi \comp \ol{c_{p}}=c$ and 
$\ol{c_{p}}(0)=p$. Such a path $\ol{c_{p}}$ is called the {\it horizontal lift} of $c$ starting from the point $p$. 
The horizontal lift always exists globally by the completeness assumption of $(M,g)$ (\cite{Her1}). 
We set 
\begin{equation}\label{HK}
k(c)p:=\ol{c_{p}}(1). 
\end{equation}
This gives a map $k(c) \colon F_{c(0)} \to F_{c(1)}$. 
Assume that each of the fiber $F_{b}$ is a totally geodesic submanifold in $M$. 
As is proved in \cite{Her1}, $k(c)$ is an isometry from $F_{c(0)}$ onto $F_{c(1)}$. 
Fix a base-point $b_{o} \in B$. 
Since $B$ is assumed to be connected, each fiber of $\pi$ is isomorphic to $F=F_{b_{o}}$. 
Set 
\begin{equation}\label{HOL}
\mr{Hol}(b_{o})=\{k(c) \mid \mbox{$c$ is a closed path with the base-point $b_{o}$}\}. 
\end{equation}
Since $k(c_{1})k(c_{2})=k(c_{2}c_{1})$ with the concatenation $c_{2} c_{1}$ of $c_{2}$ followed by $c_{1}$, 
$\mr{Hol}(b_{o})$ is a subgroup in the group of isometries, $\mr{Isom}(F)$, of $F$, 
which is called the {\it holonomy group} at $b_{o}$. 
Since $B$ is assumed to be connected, 
the group structure of the holonomy group does not depend on the choice of the base-point. 
\par
Now assume further that the horizontal distribution $\mc{H}$ is integrable. 
Then, $\mc{H}$ defines a foliation on $M$. 
For $p \in M$, let $L(p)$ denote the leaf of $\mc{H}$ through $p$. 
\begin{lem}\label{Hol1}
Let $c_{0},c_{1}$ be paths in $B$ with $c_{0}(0)=c_{1}(0)=b_{o}$ and $c_{0}(1)=c_{1}(1)=b$. 
Suppose that $c_{0}$ and $c_{1}$ are homotopic relative to $\{0,1\}$. Then $k(c_{0})=k(c_{1})$. 
\end{lem}
\begin{proof}
Let $\gamma_{s}$ ($s \in [0,1]$) be a piecewise smooth homotopy 
such that $\gamma_{s}(0)=b_{o}$, $\gamma_{s}(1)=b$, $\gamma_{0}=c_{0}$, $\gamma_{1}=c_{1}$. 
Take $p \in F$ and consider the horizontal lift $\ol{(\gamma_{s})_{p}}$ of $\gamma_{s}$ starting at $p \in F$. 
Since $\ol{(\gamma_{s})_{p}}(t)$ is a solution to a differential equation in $t$ with 
coefficients depending piecewise smoothly in $s$, 
the curve $\ell(s):=k(\gamma_{s})p=\ol{(\gamma_{s})_{p}}(1)$ in $F_{b}$ is piecewise smooth in $s$. 
Thus $\ell'(s) \in \mc{V}_{\ell(s)}$. Since $\ol{(\gamma_{s})_{p}}(t)$ is a horizontal lift, 
$\ol{(\gamma_{s})_{p}}(t) \in L(p)$ for any $t$ and $s$, and hence $\ell(s) \in L(p)$.  
Thus we have $\ell'(s) \in \mc{H}_{\ell(s)} \cap \mc{V}_{\ell(s)}=\{0\}$. 
Hence $\ell(s)$ is a constant path. Therefore, $k(c_{0})p=\ell(0)=\ell(1)=k(c_{1})p$, which shows the lemma.  
\end{proof}
From Lemma $\ref{Hol1}$, we have a surjective homomorphism
\begin{equation}\label{HOM}
\mf{h} \colon \pi_{1}(B,b_{o}) \to \mr{Isom}(F),\quad 
\mf{h}(\ispa{c}) =k(c)^{-1} \quad (\ispa{c} \in \pi_{1}(B,b_{o})), 
\end{equation}
where $\ispa{c}$ denotes the homotopy class of a closed path $c$. 
Let $\wh{p} \colon \wh{B} \to B$ denote the Riemannian universal covering of $B$, 
and identify $\pi_{1}(B,b_{o})$ with the group of covering transformations of $\wh{p}$. 
Thus, $\pi_{1}(B,b_{o})$ acts on the Riemannian product manifold $\wh{B} \times F$ as 
\[
\sigma (\alpha,p)=(\sigma \alpha, \mf{h}(\sigma)p) \quad ((\alpha,p) \in \wh{B} \times F,\ \sigma \in \pi_{1}(B,b_{o})). 
\]
Sine the action of $\pi_{1}(B,b_{o})$ on $\wh{B} \times F$ is properly discontinuous, 
the quotient manifold $\wh{B} \times_{\mf{h}} F=(\wh{B} \times F)/\pi_{1}(B,b_{o})$ has a unique 
Riemannian metric $G$ such that the natural map 
$\Pi_{\mf{h}} \colon \wh{B} \times F \to \wh{B} \times_{\mf{h}} F$ is a Riemannian covering. 
\par
To obtain a fiber product representation of $M$, 
we use the construction of the universal covering space, $\wh{B}$, of $B$ 
by the quotient space of the space of paths starting at $b_{o}$ (\cite{M}). 
Define 
\[
\mc{B}=\{c \colon [0,1] \to B : \mbox{$c$ is a path such that $c(0)=b_{o}$} \}. 
\]
For $c_{0},c_{1} \in \mc{B}$, 
write $c_{0} \sim c_{1}$ if $c_{0}$ and $c_{1}$ are homotopic with the endpoints fixed. 
We consider the quotient set $\mc{B}/{\sim}$. 
We write $\ispa{c}$ for an element in $\mc{B}/{\sim}$ determined by $c \in \mc{B}$. 
Define the map $\ol{p} \colon \mc{B}/{\sim} \to B$ by
\[
\ol{p}(\ispa{c}) =c(1) \quad (\ispa{c} \in \mc{B}/{\sim}). 
\]
Let $U$ be a connected open set in $B$ with the property that 
the inclusion $\iota \colon U \hookrightarrow B$ induces 
the trivial homomorphism from $\pi_{1}(U)$ to $\pi_{1}(B)$. 
Let $c \in \mc{B}$ be a path with $c(1) \in U$, and set 
\[
\ispa{c, U}=\{\ispa{c  c_{1}} \in \wh{B} \mid c_{1} \colon [0,1] \to U, \, c_{1}(0)=c(1)\}. 
\]
We equip $\mc{B}/{\sim}$ the topology generated by the family $\{\ispa{U,c}\}$, 
where $U$ runs over all connected open sets in $B$ with the property 
explained above and $c \in \mc{B}$ with $c(1) \in U$. 
It is well-known that the topological space $\mc{B}/{\sim}$ with the map 
$\ol{p} \colon \mc{B}/{\sim} \to B$ defined in this way is a universal covering space of $B$ (\cite{M}). 
Hence, in the rest of the paper, we identify $\wh{p} \colon \wh{B} \to B$ with $\ol{p} \colon \mc{B}/{\sim} \to B$. 
We denote $\alpha \in \wh{B}$ as $\alpha=\ispa{c_{\alpha}}$ with a 
piecewise smooth path $c_{\alpha} \in \mc{B}$. 
The action of $\pi_{1}(B,b_{o})$ on $\wh{B}$ in this formulation is given by
\[
\sigma \alpha=\ispa{\gamma_{\sigma} c_{\alpha}},\quad \sigma=\ispa{\gamma_{\sigma}} \in \pi_{1}(B,b_{o}), \ 
\alpha=\ispa{c_{\alpha}} \in \wh{B}.
\]
\begin{prop}\label{ST1}
Under the assumption made above, 
the map $\varphi_{0} \colon \wh{B} \times F \to M$ defined by 
\[
\varphi_{0} (\alpha,p) = k(c_{\alpha}) p \quad (\alpha=\ispa{c_{\alpha}} \in \wh{B},\ p \in F)
\]
descends to an isometric diffeomorphism $\varphi \colon \wh{B} \times_{\mf{h}} F \to M$  
which makes the following diagram commutative; 
\begin{equation}\label{DDD1}
\xymatrix{
\wh{B} \times_{\mf{h}} F \ar[rd]_{\pi_{\mf{h}}} \ar[rr]^-{\varphi}  & & M \ar[ld]^{\pi} \\
& B & 
}
\end{equation}
where $\pi_{\mf{h}}$ is defined as $\pi_{\mf{h}} (\Pi_{\mf{h}}(\alpha,p))=\wh{p}(\alpha)$. 
\end{prop}
Clearly $\pi_{\mf{h}}$ is well-defined, and Proposition $\ref{FPL}$ given below shows that 
$\pi_{\mf{h}}$ defines a Riemannian submersion. Thus, Proposition $\ref{ST1}$ 
shows that, as a Riemannian submersion over $B$, we can identify $\wh{B} \times_{\mf{h}} F$ and $M$ 
as a Riemannian submersion. To prove Proposition $\ref{ST1}$, we need the following lemma. 
\begin{lem}\label{ST0}
Fix $p \in F$ and consider the map $f \colon \wh{B} \to M$ defined by $f(\alpha)=k(c_{\alpha})p$. 
Take $\alpha=\ispa{c_{\alpha}} \in \wh{B}$ and $u \in T_{\alpha}\wh{B}$. Set $x=f(\alpha)$, $b=\wh{p}(\alpha)$. 
Then the tangent vector $df_{\alpha}(u) \in T_{x}M$ is the 
horizontal lift of $d\wh{p}_{\alpha}(u) \in T_{b}B$ at $x$. 
\end{lem}
\begin{proof}
We set $v = d\wh{p}_{\alpha}(u) \in T_{b}B$. 
Since $d\pi_{x}(df_{\alpha}(u))=v$, it is enough to show that $df_{\alpha}(u)$ is a horizontal vector. 
We take a path $c \colon [0,1] \to B$ such that $c(0)=b$ and $c'(0)=v$. 
We define, for each $s \in [0,1]$, the path $\gamma_{s} \colon [0,1] \to B$ by 
\[
\gamma_{s}(t)=
\begin{cases}
c_{\alpha}((1+s)t) & (0 \leq t \leq 1/(1+s)), \\
c((1+s)t-1) & (1/(1+s) \leq t \leq 1). 
\end{cases}
\]
Set $\wh{\gamma}(s):=\ispa{\gamma_{s}}$. 
Then, $\wh{\gamma}$ satisfies $\wh{\gamma}(0)=\alpha$ and $\wh{p}(\wh{\gamma}(s))=c(s)$. 
Therefore, $d\wh{p}_{\alpha}(\wh{\gamma}'(0))=c'(0)=v$. 
Since $\wh{p}$ is a local isometry, we have $\wh{\gamma}'(0)=u$. 
Let $\ol{(c_{\alpha})}_{p}$, $\ol{c_{x}}$ denote the horizontal lift of $c_{\alpha}$ starting at $p$ and that of 
$c$ starting at $x$, respectively. 
Then the horizontal lift, $\ol{(\gamma_{s})_{p}}$, of $\gamma_{s}$ on $M$ starting at $p$ is 
\[
\ol{(\gamma_{s})_{p}} (t)=
\begin{cases}
\ol{(c_{\alpha})}_{p} ((1+s)t) & (0 \leq t \leq 1/(1+s)), \\
\ol{c_{x}}((1+s)t-1) & (1/(1+s) \leq t \leq 1).
\end{cases}
\]
Thus we have $f(\wh{\gamma}(s))=\ol{(\gamma_{s})_{p}}(1)=\ol{c_{x}}(s)$.  
Differentiate this at $s=0$, we see that $df_{\alpha}(u)=\ol{c_{x}}'(0)$ is the horizontal lift of $v$. 
This shows the lemma. 
\end{proof}
\par
\vspace{10pt}
\noindent{\bf Proof of Proposition \ref{ST1}: } 
First of all, note that the map $\varphi_{0}$ is smooth because $k(c_{\alpha})p=\ol{(c_{\alpha})_{p}}(1)$, 
where $\ol{(c_{\alpha})_{p}}$ is the horizontal lift of $c_{\alpha}$ starting at $p$, is defined as a 
solution to ODE whose coefficients depends smoothly on the endpoints of $c_{\alpha}(1)$, and 
the solution depends smoothly on the initial condition $p$. 
Let $(\alpha,p) \in \wh{B} \times F$ and let $\sigma=\ispa{\gamma_{\sigma}} \in \pi_{1}(B,b_{o})$. Then 
\[
\varphi_{0}(\sigma \alpha, \mf{h}(\sigma) p) =k(\gamma_{\sigma} c_{\alpha}) \mf{h}(\sigma)p
=k(\gamma_{\sigma} c_{\alpha}) h(\gamma_{\sigma}^{-1}) p =
k(\gamma_{\sigma}^{-1} \gamma_{\sigma} c_{\alpha})p =k(c_{\alpha}) p =\varphi_{0}(\alpha,p). 
\]
Thus, $\varphi_{0}$ defines a map $\varphi \colon \wh{B} \times_{\mf{h}} F \to M$. 
Suppose $(\alpha,p), (\beta,q) \in \wh{B} \times F$ satisfies 
$\varphi(\Pi_{\mf{h}}(\alpha,p))=\varphi(\Pi_{\mf{h}}(\beta,q))$, 
that is, $k(c_{\alpha})p=k(c_{\beta})q$. Projecting this on $B$ by $\pi$, we see $c_{\alpha}(1)=c_{\beta}(1)$. 
Hence $c_{\alpha}c_{\beta}^{-1}$ is a closed path with the base-point $b_{o}$. 
Set $\sigma=\ispa{c_{\alpha}c_{\beta}^{-1}} \in \pi_{1}(B,b_{o})$. 
Then, $\sigma \beta =\alpha$ and $\mf{h}(\sigma)q=p$. 
Thus, $\Pi_{\mf{h}}(\alpha,p)=\Pi_{\mf{h}}(\beta,q)$, and hence $\varphi$ is injective. 
Next, take $x \in M$ and set $b=\pi(x)$. Take a path $c$ in $B$ from $b$ to $b_{o}$. 
Set $\alpha=\ispa{c^{-1}} \in \wh{B}$ and $p=k(c)x \in F$. 
Then, by definition, $\varphi_{0}(\alpha,p)=x$ and hence $\varphi$ is surjective. 
Finally we shall show that $d\varphi_{0}$ preserves the Riemannian structures. 
Take $(\alpha,p) \in \wh{B} \times F$ and set $x = \varphi_{0}(\alpha, p)=k(c_{\alpha})p \in M$ and $b=\wh{p}(\alpha) \in B$. 
Let $(u,v) \in T_{\alpha}\wh{B} \times T_{p}F$. We set $u'=(d\varphi_{0})_{(\alpha,p)}(u,0)$, 
$v'=(d\varphi_{0})_{(\alpha,p)}(0,v)$ so that $(d\varphi_{0})_{(\alpha,p)}(u,v)=u'+v' \in T_{x} M$. 
Since $k(c_{\alpha}) \colon F \to \pi^{-1}(b)$ is an isometric diffeomorphism, $v$ and 
$v'=[dk(c_{\alpha})]_{p}v \in T_{x}\pi^{-1}(b)$ have the same length. 
From Lemma $\ref{ST0}$, it follows that $u'=(d\varphi_{0})_{(\alpha,p)}(u,0)$ is 
the horizontal lift of $d\wh{p}_{\alpha}(u)$. 
Therefore $u$ and $u'$ has the same length. Since $u' \in \mc{H}_{x}$ and $v' \in \mc{V}_{x}$, 
these are perpendicular. Thus $d\varphi_{0}$, and hence $d\varphi$, preserves the Riemannian structures. 
This shows that $\varphi$ is an isometry from $\wh{B} \times_{\mf{h}} F$ onto $M$. 
\hfill$\blacksquare$
\par
\vspace{10pt}
In general, consider a regular Riemannian covering $\wt{p} \colon \wt{B} \to B$ with 
the covering transformation group $G(\wt{p})$. Fix $x_{o} \in \wt{p}^{-1}(b_{o})$. 
Let $\mf{p} \colon \wh{B} \to \wt{B}$ be the lift of $\wh{p}$ with $\mf{p}(e_{b_{o}})=x_{o}$.  
Since $\wt{p}$ is regular, 
the map 
\[
\mu(\sigma) \colon \wt{B} \to \wt{B},\quad 
\mu(\sigma) x=\mf{p}(\sigma \alpha) \quad (x \in \wt{B}, \ \alpha \in \wh{B},\ \mf{p}(\alpha)=x)
\]
is well-defined and defines a left action of $\pi_{1}(B,b_{o})$ on $\wt{B}$, 
and it gives a surjective homomorophism $\mu \colon \pi_{1}(B,b_{o}) \to G(\wt{p})$. 
\begin{prop}\label{FPL}
Let $(B,j)$, $(F,q)$ be compact connected Riemannian manifolds. Fix a base-point $b_{o} \in B$. 
Let $\wt{p} \colon \wt{B} \to B$ be a regular Riemannian covering and let 
$\phi \colon G(\wt{p}) \to \mr{Isom}(F)$ be a homomorphism. 
Let $\Pi_{\phi} \colon \wt{B} \times F \to \wt{B} \times_{\phi} F$ denote the natural projection. 
Equip $\wt{B} \times_{\phi} F$ the unique Riemannian metric $g$ such that 
$\Pi_{\phi}^{*}g=\wh{p}^{*}j \times q$. Then the map 
\[
\pi_{\phi} \colon \wt{B} \times_{\phi} F \to B,\quad \pi_{\phi} (\Pi_{\phi} (x, p))
=\wt{p}(x) \quad ((x, p) \in \wt{B} \times F) 
\]
is well-defined and is a Riemannian submersion. We have the following. 
\begin{enumerate}
\item[{\rm (i)}] $\pi_{\phi}$ has totally geodesic fibers; 
\item[{\rm (ii)}] The horizontal distribution of $\pi_{\phi}$ is integrable; 
\item[{\rm (iii)}] The holonomy representation 
$\mf{h} \colon \pi_{1}(B,b_{o}) \to \mr{Isom}(\pi_{\phi}^{-1}(b_{o}))$ of $\pi_{\phi}$ 
is given by $\mf{h}(\sigma)=f \comp \phi \comp \mu (\sigma) \comp f^{-1}$, where 
$f$ is the isometry from $F$ onto $\pi_{\phi}^{-1}(b_{o})$ given by 
\[
f(p)=\Pi_{\phi}(x_{o}, p)\quad (p \in F), 
\]
and $\mu \colon \pi_{1}(B,b_{o}) \to G(\wt{p})$ is the surjective homomorphism defined above. 
\end{enumerate}
\end{prop}
%
%
\begin{proof}
It is clear that $\pi_{\phi}$ is well-defined. Take $(x,p) \in \wt{B} \times F$. 
Differentiating $\pi_{\phi} \comp \Pi_{\phi}(x,p)=\wt{p}(x)$, we have 
\begin{equation}\label{DFE1}
(d\pi_{\phi})_{\Pi_{\phi}(x,p)} \left(
(d\Pi_{\phi})_{(x,p)}(w,v)
\right)=(d\wt{p})_{x}(w),\quad 
(w,v) \in T_{x}\wt{B} \times T_{p}F.
\end{equation}
From this it follows that $\mr{ker}(d\pi_{\phi})_{\Pi_{\phi}(x,p)} =(d\Pi_{\phi})_{(x,p)}(T_{p}F)$ 
and its orthogonal complement is $\mc{H}_{\Pi_{\phi}(x,p)}=(d\Pi_{\phi})_{(x,p)} (T_{x}\wt{B})$. 
Since $\wt{p}$ is a Riemannian covering, it follows from $\eqref{DFE1}$ that \break
$(d\pi_{\phi})_{\Pi_{\phi}(x,p)} \colon \mc{H}_{\Pi_{\phi}(x,p)} \to T_{b}B$ is an isometric isomorphism, where 
$b=\pi_{\phi}(x)=\wt{p}(x)$. 
Thus, $\pi_{\phi} \colon \wt{B} \times_{\phi} F \to B$ is a Riemannian submersion. 
%
%
The assertion (i) is clear from the fact that $\{x\} \times F$ is totally geodesic in $\wt{B} \times F$. 
The integrability of $\mc{H}$ in (ii) follows from the fact that 
$\Pi_{\phi} \colon \wt{B} \times F \to \wt{B} \times_{\phi} F$ 
is a local isometry and hence $\wt{B} \times_{\phi} F$ is locally isomorphic to the Riemannian product 
of $\wt{B}$ and $F$. 
Let us show (iii). 
Let $\gamma \colon [0,1] \to B$ be a smooth path from $b_{o}$ to $b$. Take $x \in \wt{p}^{-1}(b_{o})$ and $p \in F$. 
Let $\wt{\gamma} \colon [0,1] \to \wt{B}$ be the lift of $\gamma$ with $\wt{\gamma}(0)=x$. 
Then the path $\ol{\gamma} \colon [0,1] \to \wt{B} \times_{\phi} F$ defined by 
$\ol{\gamma}(t)=\Pi_{\phi}(\wt{\gamma}(t), p)$ 
is the horizontal lift of $\gamma$ starting at $\Pi_{\phi}(x,p)$. In particular, we have 
\[
k(\gamma)\Pi_{\phi}(x,p)=\Pi_{\phi}(\wt{\gamma}(1), p). 
\]
Take $\alpha=\ispa{c_{\alpha}} \in \wh{B}$ such that $\mf{p}(\alpha)=x$. 
Then the computation similar to that in the proof of Lemma $\ref{ST0}$ and the above formula give 
\begin{equation}\label{HGF}
k(\gamma) \Pi(x,p) =\Pi (\mf{p}(\ispa{c_{\alpha}\gamma}), p). 
\end{equation}
Now let us consider the case $b=b_{o}$ and $x=x_{o}$. In this case we can take $\alpha=e_{b_{o}}=\ispa{b_{o}^{\#}}$. 
Then $\ispa{c_{\alpha}\gamma}=\sigma$, where we set $\sigma=\ispa{\gamma} \in \pi_{1}(B,b_{o})$. 
It follows from the formula $\eqref{HGF}$ that 
\[
\begin{split}
k(\gamma^{-1}) f(p) & =k(\gamma^{-1}) \Pi(x_{o},p)=\Pi(\mf{p}(\ispa{\gamma^{-1}}), p) =\Pi(\mu(\sigma^{-1}) x_{o}, p) \\
& =\Pi(x_{o}, \phi(\mu(\sigma))p) = f(\phi(\mu(\sigma))p). 
\end{split}
\]
Since $\mf{h}(\sigma)=k(\gamma^{-1})$, we conclude the assertion in (iii). 
\end{proof}
Let us turn to our Riemannian submersion $\pi \colon M \to B$ and the 
holonomy representation $\mf{h}$ given in $\eqref{HOM}$. Define 
\begin{equation}\label{HC1}
\varpi_{H} \colon B_{H}=\wh{B}/\mr{ker}(\mf{h}) \to B,\ \ \varpi_{H}(\mr{ker}(\mf{h})\alpha)=\wh{p}(\alpha). 
\end{equation}
Equip $B_{H}$ the Riemannian metric $\varpi_{H}^{*}j$ so that $\varpi_{H}$ is a regular Riemannian covering. 
The group of covering transformations of $\varpi_{H}$ is isomorphic to 
$\pi_{1}(B,b_{o})/\mr{ker}(\mf{h}) \cong \mr{Hol}(b_{o})$. 
Denote $\iota_{H} \colon \mr{Hol}(b_{o}) \hookrightarrow \mr{Isom}(F)$ the natural inclusion. 
Then, applying Proposition $\ref{FPL}$, the manifold $B_{H} \times_{\iota_{H}} F$ with the metric induced 
from the product metric on $B_{H} \times F$ becomes a Riemannian submersion over $B$ 
with totally geodesic fibers and the integrable horizontal distribution. 
We have the following. 
\begin{lem}\label{CHL}
The map $\Phi \colon \wh{B} \times F \to B_{H} \times F$ defined by $\Phi(\alpha,p)=(p_{H}(\alpha), p)$ gives 
an isometric diffeomorphism $\varphi \colon M \cong \wh{B} \times_{\mf{h}} F \to B_{H} \times_{\iota_{H}} F$ 
such that the following diagram commutative; 
\[
\xymatrix{
\wh{B} \times_{\mf{h}} F \ar[rd]_{\pi_{\mf{h}}} \ar[rr]^-{\varphi}  & & B_{H} \times_{\iota_{H}}F \ar[ld]^{\pi_{\iota_{H}}} \\
& B & 
}
\]
\end{lem}
\begin{proof}
The map $\varphi$ is defined so that the diagram 
\[
\xymatrix{
\wh{B} \times F \ar[r]^{\Phi} \ar[d]_{\Pi_{\mf{h}}} & B_{H} \times F \ar[d]^{\Pi_{\iota_{H}}} \\
\wh{B} \times_{\mf{h}} F \ar[r]_-{\varphi} & B_{H} \times_{\iota_{H}} F
}
\]
commutative. Thus $\varphi$ is a local isometry. It is clear that $\varphi$ is surjective. 
Suppose $\varphi(\Pi_{\mf{h}}(\alpha, p))=\varphi(\Pi_{\mf{h}}(\beta, q))$. 
Then $\Pi_{\iota_{H}}(p_{H}(\alpha), p)=\Pi_{\iota_{H}}(p_{H}(\beta), q)$, and hence there exists 
$\sigma \in \pi_{1}(B,b_{o})$ such that $p_{H}(\beta)=\mf{h}(\sigma) p_{H}(\alpha)$ and $q=\mf{h}(\sigma)p$. 
Since $\mf{h}(\sigma) p_{H}(\alpha)=p_{H}(\sigma \alpha)$, 
there exists $s \in \mr{ker}(\mf{h})$ such that $\beta=s \sigma \alpha$. 
Since $\mf{h}(s\sigma)p=\mf{h}(\sigma)p=q$, we have $(s\sigma)(\alpha,p)=(\beta,q)$.  
Hence $\varphi$ is injective and the assertion follows. 
\end{proof}
Summarizing the above discussion, we can identify three Riemannian submersions 
$\pi \colon M \to B$, $\pi_{\mf{h}} \colon \wh{B} \times_{\mf{h}} F \to B$ 
and $\pi_{\iota_{H}} \colon B_{E} \times_{\iota_{H}} F \to B$. 
The identifications are given by 
\begin{equation}\label{ID}
\begin{split}
& \wh{B} \times_{\mf{h}} F \ni \Pi_{\mf{h}}(\alpha,p) \mapsto k(c_{\alpha})p \in M, \\
& B_{H} \times_{\iota_{H}} F \ni \Pi_{\iota_{H}} (p_{H}(\alpha),p) \mapsto k(c_{\alpha})p \in M, 
\end{split}
\end{equation}
where $\alpha = \ispa{c_{\alpha}} \in \wh{B}$. 
\subsection{Unitary operator $W$ and its properties}\label{Uni}
As in Introduction, let $\pi \colon (M,g) \to (B,j)$ be compact connected Riemannian submersion 
such that each fiber $F_{b}$ over $b \in B$ is connected and totally geodesic in $M$ and that 
the horizontal distribution $\mc{H}$ is integrable. Fix the base-point $b_{o} \in B$ and set $F=F_{b_{o}}$. 
In what follows we identify $\pi \colon M \to B$ with the fiber product 
$\pi_{\iota_{H}} \colon B_{H} \times_{\iota_{H}} F \to B$. 
Thus we write simply $\pi$ for $\pi_{\iota_{H}}$ and $\Pi$ for $\Pi_{\iota_{H}}$. 
The flat vector bundle $\pi_{\rho} \colon \mc{E}_{\rho} \to B$ in $\eqref{FVB1}$ is defined as 
the quotient $\mc{E}_{\rho}=(B_{H} \times L^{2}(F))/{\sim}$, where, for 
$(x_{1},f_{1}), (x_{2},f_{2}) \in B_{H} \times L^{2}(F)$, 
$(x_{1},f_{1}) \sim (x_{2},f_{2})$ if and only if 
$(x_{2},f_{2})=(h x_{1}, \rho(h)f_{1})$ with some $h \in \mr{Hol}(b_{o})$. 
\par
Let $s \colon B \to \mc{E}_{\rho}$ be a smooth section of $\mc{E}_{\rho}$. 
It is well-known that there exists a unique smooth $\rho$-equivariant map $\psi_{s} \colon B_{H} \to L^{2}(F)$ 
such that $s(b)=[x, \psi_{s}(x)]$ for any $x \in B_{H}$ with $\varpi_{H}(x)=b$. 
By the correspondence $s \leftrightarrow \psi_{s}$, we identify the space of smooth sections 
$C^{\infty}(\mc{E}_{\rho})$ with the space of smooth $\rho$-equivariant maps from $B_{H}$ to $L^{2}(F)$. 
Let $D$ be an open fundamental domain of the covering $\varpi_{H} \colon B_{H} \to B$. 
The norm $\|s\|_{L^{2}(\mc{E}_{\rho})}$ is, in terms of the map $\psi_{s}$, given by 
\begin{equation}\label{Norm1}
\|s\|_{L^{2}(\mc{E}_{\rho})} =\int_{D} \|\psi_{s}(x)\|_{L^{2}(F)}^{2}\,d\mu_{B_{H}}(x). 
\end{equation}
The integrand $\|\psi_{s}(x)\|_{L^{2}(F)}^{2}$ of the integral in the left-hand side 
is invariant under the action of $\mr{Hol}(b_{o})$ and 
hence is regarded as a function on $B$. Thus the integral on $D$ equals that on $B$. 
Any $f \in C^{\infty}(M)$ corresponds to 
a smooth function $\Pi^{*}f$ on $B_{H} \times F$ invariant under the action of $\mr{Hol}(b_{o})$. 
Therefore, we have the following map; 
\begin{equation}\label{UED}
W \colon C^{\infty}(M) \to C^{\infty}(\mc{E}_{\rho}),\quad (Wf)(x)(p) = f(\Pi(x,p)) \quad 
((x,p) \in B_{H} \times F). 
\end{equation}
The restriction of $\varpi_{H}$ on $D$ is an isometry onto a connected open set $U$ in $B$. 
Since the covering $\varpi_{H}$ is regular, the fiber $F_{b}$ over $b \in U$ can be written as 
\[
F_{b}=\{\Pi(x_{b},p) \mid p \in F\},\quad x_{b}=\varpi_{H}|_{D}^{-1}(b). 
\]
From this it follows that the $L^{2}$-norm of $f \in C^{\infty}(M)$ can be written as 
\[
\|f\|_{L^{2}(M)}^{2}=\int_{B} \int_{F} |f(\Pi(x_{b},p))|^{2}\,d\mu_{F} d\mu_{B}
=\int_{D} \|Wf (x) \|_{L^{2}(F)}^{2}\,d\mu_{B_{H}}(x) =\|Wf\|_{L^{2}(\mc{E}_{\rho})}^{2}. 
\]
The image $WC^{\infty}(M)$ is the subspace $C^{\infty}_{\mr{s}}(\mc{E}_{\rho})$ of $C^{\infty}(\mc{E}_{\rho})$ given by 
\[
C^{\infty}_{\mr{s}} (\mc{E}_{\rho}) =
\{\phi \in  C^{\infty}(\mc{E}_{\rho}) \mid B_{H} \times F \ni (x,p) \mapsto \phi(x)(p) \mbox{ is smooth}\}. 
\]
\begin{lem}\label{UEL3}
$W$ is extended to a unitary operator from $L^{2}(M)$ onto $L^{2}(\mc{E}_{\rho})$. 
\end{lem}
\begin{proof}
Since $L^{2}(\mc{E}_{\rho})$ is the completion of $C^{\infty}(\mc{E}_{\rho})$ with respect to the $L^{2}$-norm, 
it is enough to show that $C^{\infty}_{\mr{s}}(\mc{E}_{\rho})$ is dense in $L^{2}(\mc{E}_{\rho})$. 
Take $\ve>0$, $\xi \in L^{2}(\mc{E}_{\rho})$, and $\phi \in C^{\infty}(\mc{E}_{\rho})$ such that 
$\|\xi -\phi\|_{L^{2}(\mc{E}_{\rho})} <\ve/2$. 
For any $\mu \in \spec(\Delta_{F})$, let $P_{\mu}$ denote the orthogonal projection on $L^{2}(F)$ onto 
the eigenspace $E(\mu)$ corresponding to the eigenvalue $\mu$ of $\Delta_{F}$. 
Note that $E(\mu)$ is a finite dimensional subspace in $C^{\infty}(F)$. 
For any $\lambda>0$ and $x \in B_{H}$, we set 
\[
\phi_{\lambda}(x)=\sum_{ \mu \in \spec(\Delta_{F}),\, \mu \leq \lambda} P_{\mu}\phi(x). 
\]
Let $f_{1},\ldots,f_{m(\mu)}$ be an orthonormal basis of $E(\mu)$ with $m(\mu)=\dim E(\mu)$. 
Then, 
\[
P_{\mu}\phi(x)=\sum_{j=1}^{m(\mu)} \ispa{\phi(x), f_{j}}_{L^{2}(F)} f_{j}. 
\]
Since $\phi \colon B_{H} \to L^{2}(F)$ is smooth, the function $\ispa{\phi(x), f_{j}}_{L^{2}(F)}$ 
on $B_{H}$ is smooth. Since $f_{j} \in C^{\infty}(F)$, we see that 
the function $(x,p) \mapsto \phi_{\lambda}(x)(p)$ is smooth on $B_{H} \times F$. 
Since each $h \in \mr{Hol}(b_{o})$ is an isometry on $F$, $\rho(h)$ 
commutes with $P_{\mu}$, and the $\mr{Hol}(b_{o})$-equivariance of $\phi$ implies 
\[
\begin{split}
\phi_{\lambda}(hx) & =\sum_{\mu \in \spec(\Delta_{F}),\, \mu \leq \lambda}P_{\mu}\phi(hx) 
=  \sum_{\mu \in \spec(\Delta_{F}),\, \mu \leq \lambda}P_{\mu}\rho(h)\phi(x)  \\
& = \sum_{\mu \in \spec(\Delta_{F}),\, \mu \leq \lambda}\rho(h)P_{\mu}\phi(x) =\rho(h)\phi_{\lambda}(x).  
\end{split}
\]
This shows $\phi_{\lambda} \in C^{\infty}_{\mr{s}}(\mc{E}_{\rho})$. 
Since $\|\phi_{\lambda}(x)\|_{L^{2}(F)} \leq \|\phi(x)\|_{L^{2}(F)}$ and 
$\|\phi_{\lambda}(x) -\phi(x)\|_{L^{2}(F)} \to 0$ as $\lambda \to \infty$ for each $x \in B_{H}$, 
the Lebesgue convergence theorem shows that $\|\phi_{\lambda} -\phi\|_{L^{2}(\mc{E}_{\rho})} \to 0$ as 
$\lambda \to \infty$. Take $\lambda$ large enough so that $\|\phi_{\lambda} -\phi\|_{L^{2}(\mc{E}_{\rho})} <\ve/2$. 
Then 
\[
\|\xi -\phi_{\lambda}\|_{L^{2}(\mc{E}_{\rho})} \leq \|\xi - \phi\|_{L^{2}(\mc{E}_{\rho})} +
\|\phi -\phi_{\lambda}\|_{L^{2}(\mc{E}_{\rho})} \leq \ve, 
\]
which shows that $C^{\infty}_{\mr{s}}(\mc{E}_{\rho})$ is dense in $L^{2}(\mc{E}_{\rho})$. 
\end{proof}
$L^{2}(F)$ has a decomposition into the eigenspaces of the Laplacian on $F$, 
\[
L^{2}(F)=\bigoplus_{\mu \in \spec(\Delta_{F})} E(\mu). 
\]
The representation $\rho$ of $\mr{Hol}(b_{o})$ is decomposed into 
the subrepresentations determined by restricting it to the eigenspaces $E(\mu)$, 
for which we write
\[
\rho = \bigoplus_{\mu \in \spec(\Delta_{F})} \rho_{\mu}, \quad 
\rho_{\mu} \colon \mr{Hol}(b_{o}) \to U(E(\mu)), \quad 
\rho_{\mu}(h)=\rho(h)|_{E(\mu)} \quad (h \in \mr{Hol}(b_{o})). 
\]
By this decomposition, the flat vector bundle $\mc{E}_{\rho}$ has the natural decomposition 
\[
\mc{E}_{\rho}=
\bigoplus_{\mu \in \spec(\Delta_{F})} \mc{E}_{\rho_{\mu}}
\]
into the finite-rank flat vector bundles $\mc{E}_{\rho_{\mu}}=B_{H} \times_{\rho_{\mu}} E(\mu)$. 
\begin{lem}\label{UED2}
The unitary operator $W \colon L^{2}(M) \to L^{2}(\mc{E}_{\rho})$ defined in Lemma $\ref{UEL3}$ satisfies 
\begin{equation}\label{CM1}
W C^{\infty}(M,\mu) =C^{\infty} (\mc{E}_{\rho_{\mu}}). 
\end{equation}
\end{lem}
\begin{proof}
First, note that $C^{\infty}_{\mr{s}} (\mc{E}_{\rho_{\mu}})$ coincides with $C^{\infty}(\mc{E}_{\rho_{\mu}})$. 
Take $\alpha=\ispa{c_{\alpha}} \in \wh{B}$. 
Since $k(c_{\alpha}) \colon F \to F_{c_{\alpha}(1)}$ is an isometry, we have 
\[
\Delta_{F} \comp k(c_{\alpha})^{*} = k(c_{\alpha})^{*} \comp \Delta_{F_{c_{\alpha}(1)}}. 
\]
Take $f \in C^{\infty}(M,\mu)$ and write $f_{b}=f|_{F_{b}}$ for each $b \in B$. Then, 
for $p \in F$, 
\[
\begin{split}
(\Delta_{F} k(c_{\alpha})^{*}f_{c_{\alpha}(1)}) (p) 
& = (k(c_{\alpha})^{*}\Delta_{F_{c_{\alpha}(1)}} f_{c_{\alpha}(1)}) (p) 
 = (\Delta_{F_{c_{\alpha}(1)}} f_{c_{\alpha}(1)}) (k(c_{\alpha})p) \\
& = (\Delta_{V} f) (k(c_{\alpha})p) 
= \mu f(k(c_{\alpha})p) =\mu (k(c_{\alpha})^{*}f_{c_{\alpha}(1)} )(p). 
\end{split}
\]
In terms of the isometry $k(c_{\alpha})$, the section $Wf$ is, as an equivariant map from $B_{H}$ to $L^{2}(F)$, 
written in the form $(Wf)(x)(p)=f(k(c_{\alpha})p)$ with $\alpha=\ispa{c_{\alpha}} \in \wh{B}$, $x=p_{H}(\alpha) \in B_{H}$. 
Hence, the above computation shows that the function $(Wf)(x)=k(c_{\alpha})^{*}f_{c_{\alpha}(1)}$ on $F$ 
is an eigenfunction of $\Delta_{F}$ with the eigenvalue $\mu$. Therefore, 
we have $Wf \in C^{\infty}(\mc{E}_{\rho_{\mu}})$. 
Conversely, take $s \in C^{\infty}(\mc{E}_{\rho_{\mu}})$ 
and write $\phi_{s} \colon B_{H} \to E(\mu)$ the $\rho_{\mu}$-equivariant 
map corresponding to the section $s$. Regarding $\phi_{s}$ as an $\rho$-equivariant map 
from $B_{H}$ to $L^{2}(F)$, we take $f \in C^{\infty}(M)$ with $\phi_{s}=Wf$. 
Then, for $\alpha=\ispa{c_{\alpha}} \in \wh{B}$, we have $(Wf)(p_{H}(\alpha)) \in E(\mu)$ and hence 
$\Delta_{F}[ Wf(p_{H}(\alpha))]=\mu Wf(p_{H}(\alpha))$. Thus, for $p \in F$, 
\[
\begin{split}
\mu Wf (p_{H}(\alpha)) (p) & 
= \Delta_{F} k(c_{\alpha})^{*}f_{c_{\alpha}(1)} (p) = k(c_{\alpha})^{*} \Delta_{F_{c_{\alpha}(1)}} f_{c_{\alpha}(1)} (p) \\
& = (\Delta_{F_{c_{\alpha}(1)}} f_{c_{\alpha}(1)}) (k(c_{\alpha}) p) 
= (\Delta_{V}f) (k(c_{\alpha})p)  =(W\Delta_{V}f) (p_{H}(\alpha))(p). 
\end{split}
\]
This shows the formula $\mu W(f) =W(\Delta_{V}f)$, and hence $\mu f =\Delta_{V}f$. 
Therefore we conclude $f \in C^{\infty}(M,\mu)$ and obtain $\eqref{CM1}$. 
\end{proof}
\subsection{Twisted Laplacian and proof of Theorem \ref{MN1}} \label{Twi} 
We recall the definition of the twisted Laplacian acting on $C^{\infty}(\mc{E}_{\rho})$. 
Any vector field $X$ on $B$ defines a first-order differential operator, $X^{\rho}$, 
acting on $C^{\infty}(\mc{E}_{\rho})$. To define $X^{\rho}$, let $X^{H}$ denote 
the lift of $X$ on $B_{H}$. Let $s \in C^{\infty}(\mc{E}_{\rho})$ and let $\phi_{s} \colon B_{H} \to L^{2}(F)$ 
be the $\rho$-equivariant smooth map corresponding to $s$. 
Then the function $X^{H} \phi_{s}$ is also $\rho$-equivariant, 
and thus defines an element in $C^{\infty}(\mc{E}_{\rho})$ which is denoted by $X^{\rho}s$. 
By using this, the twisted Laplacian, $\Delta_{\rho}$, is defined as 
\begin{equation}\label{TLAP}
(\Delta_{\rho}s)(b)=-\sum_{k}(X_{k}^{\rho}X_{k}^{\rho}s)(b) +((\nabla^{B}_{X_{k}}X_{k})^{\rho}s)(b) \quad 
(s \in C^{\infty}(\mc{E}_{\rho}), b \in B), 
\end{equation}
where $\{X_{1},\ldots,X_{p}\}$ is a local orthonormal frame of the tangent bundle $TB$ of $B$ 
on a neighborhood of $b \in B$ and $\nabla^{B}$ is the Levi-Civita connection on $B$. 
It follows from this definition that the $\rho$-equivariant smooth map corresponding to 
$\Delta_{\rho}s$ is $\Delta_{B_{H}}\phi_{s}$ where $\phi_{s}$ is the $\rho$-equivariant smooth map 
corresponding to $s$. 
\begin{lem}\label{UEL4}
Let $X$ be a vector field on $B$ and let $\ol{X}$ denote the horizontal lift of $X$ to $M$. 
Then, for any $f \in C^{\infty}(M)$, we have $W\ol{X}f=X^{\rho}Wf$. 
\end{lem}
\begin{proof}
Identifying $M$ with $B_{H} \times_{\iota_{H}} F$ as is given in $\eqref{ID}$, 
$f \in C^{\infty}(M)$ can be identified with the function $\Pi^{*}f$ on $B_{H} \times F$ 
invariant under the action of $\mr{Hol}(b_{o})$. 
Let $X^{H}$ denote the lift of $X$ on $B_{H}$ and consider $X^{H}$ as 
a vector field on $B_{H} \times F$ in an obvious way. 
Then, by the proof of Proposition $\ref{FPL}$, we have $\ol{X}=d\Pi (X^{H})$. 
Therefore, we see $\Pi^{*}\ol{X}f = X^{H} \Pi^{*}f$, and hence $W \ol{X}f = X^{\rho} Wf$. 
\end{proof}
\noindent{\bf Proof of Theorem \ref{MN1}:}\hspace{1pt} 
We first show that $\Delta_{\rho}Wf=W \Delta_{H}f$ for $f \in C^{\infty}(M)$. Take $b \in B$ and 
an orthonormal frame $\{X_{k}\}_{k=1}^{p}$ of the tangent bundle of $B$ near $b$. 
Then, 
\[
\Delta_{H}f =- \sum_{k} \ol{X}_{k}\ol{X}_{k}f + \sum_{k}\nabla^{M}_{\ol{X}_{k}}\ol{X}_{k} f, 
\]
where $\nabla^{M}$ is the Levi-Civita connection of $(M,g)$. 
Since the horizontal lift of $\nabla^{B}_{X}X$ is $\nabla^{M}_{\ol{X}}\ol{X}$, 
it follows from Lemma $\ref{UEL4}$ that $\Delta_{\rho}W=W\Delta_{H}$ on $C^{\infty}(M)$. 
Since $\Delta_{H}$ has an $L^{2}$-orthonormal basis consisting of smooth eigenfunction, so does $\Delta_{\rho}$. 
Hence the symmetric operator $\Delta_{\rho}$ on $C^{\infty}(\mc{E}_{\rho})$ is essentially self-adjoint. 
Let $\ol{\Delta_{H}}$ be the closure of $\Delta_{H}$ with domain $\mc{D}(\ol{\Delta_{H}})$. 
Then, the same method as in the proof of Lemma $\ref{UEL3}$ shows that any $s \in C^{\infty}(\mc{E}_{\rho})$ 
can be approximated in the $L^{2}$-sense by sections $s_{\lambda} \in C^{\infty}_{\mr{s}}(\mc{E}_{\rho})$ 
such that $\Delta_{\rho}s$ is also approximated in the $L^{2}$-sense by $\Delta_{\rho}s_{\lambda}$. 
Therefore, $W\mc{D}(\ol{\Delta_{H}})$ contains $C^{\infty}(\mc{E}_{\rho})$. 
Thus $W \ol{\Delta_{H}} W^{*}$ is a self-adjoint extension of $\Delta_{\rho}$, and 
we have $W \ol{\Delta_{H}} W^{*} =\ol{\Delta_{\rho}}$, which conclude the assertion of Theorem $\ref{MN1}$. 
\hfill$\blacksquare$
\subsection{Proof of Corollary $\ref{TEN}$} \label{Pro}
We show $\Delta_{\rho}|_{C^{\infty} (\mc{E}_{\rho_{\mu}})} =\Delta_{\rho_{\mu}}$. 
Let $s \in C^{\infty}(\mc{E}_{\rho_{\mu}})$. 
Take $f \in C^{\infty}(M, \mu)$ with $Wf=s$. 
Since $\Delta_{V}\Delta_{H}=\Delta_{H}\Delta_{V}$, $\Delta_{H}f$ is contained in $C^{\infty}(M, \mu)$. 
Thus we have $W\Delta_{H}f \in C^{\infty}(\mc{E}_{\rho_{\mu}})$. 
Hence $\Delta_{\rho}s=\Delta_{\rho}Wf=W\Delta_{H}f \in C^{\infty} (\mc{E}_{\rho_{\mu}})$. 
According to the proof of Lemma $\ref{UEL4}$, 
we have $W\ol{X}f=X^{\rho_{\mu}}Wf$ for any vector field $X$ on $B$ and $f \in C^{\infty}(M, \mu)$, 
because $\rho_{\mu}$ is the restriction of $\rho$ to $E(\mu)$. 
Therefore, $W\Delta_{H}f=\Delta_{\rho_{\mu}}Wf$ for $f \in C^{\infty}(M, \mu)$. 
Hence $\Delta_{\rho}s=\Delta_{\rho_{\mu}}s$ for $s \in C^{\infty}(\mc{E}_{\rho_{\mu}})$. 
Since $E(\mu)$ is of finite-dimensional, the general spectral theory for elliptic operators shows Corollary $\ref{TEN}$. 
\hfill$\blacksquare$

\section{Spectral problem for the canonical variation} \label{CV}
This section is devoted to a proof of Theorem $\ref{BBC}$. 
Theorem $\ref{BBC}$ follows from Theorem $\ref{MN1}$ and Sunada's inequality (\cite{Su1}) 
with an observation concerning about the spectral gap 
\begin{equation}\label{SGD}
\msc{sg}=\inf_{0 \neq u \in \mathbf{1}^{\perp} \cap C^{\infty}(M)} 
\frac{\ispa{\Delta_{H}u,u}_{L^{2}(M)}}{\|u\|_{L^{2}(M)}^{2}}
\end{equation}
of the horizontal Laplacian, where $\mathbf{1}^{\perp}$ is the orthogonal complement 
in $L^{2}(M)$ of the space of constant functions. 
Since $d\mu_{(M,g_{t})}=t^{n-p} d\mu_{(M,g)}$, 
where $n-p=\dim F$, the $L^{2}$-inner product 
given by $g_{t}$ is essentially the same with that given by $g$, 
and thus, in the following, we use the $L^{2}$-inner product defined by the metric $g$. 
\subsection{Spectral gap of horizontal Laplacian}
It is straightforward to see that $\msc{sg}>0$ if and only if zero is a simple eigenvalue 
of $\Delta_{H}$ and an isolated point in $\spec(\Delta_{H})$. 
Therefore, according to \cite{BBB}, 7.7 Proposition, 
$\dsp \lim_{t \to \infty}\Lambda_{1}(M,g_{t})=+\infty$ if $\msc{sg}>0$. 
The following is a direct relationship between $\Lambda_{1}(M,g_{t})$ (or 
more precisely $\lambda_{1}(g_{t})$) and $\msc{sg}$. 
Lemma $\ref{SG}$ below holds without the assumption on the integrability of 
the horizontal distribution $\mc{H}$ but the following proof uses the assumption on the totally 
geodesic fibration. 
\begin{lem}\label{SG}
We have 
\[
\lim_{t \to \infty} \lambda_{1}(g_{t}) =\msc{sg}. 
\]
\end{lem}
\begin{proof}
Let $u_{t} \in C^{\infty}(M)$ be an eigenfunction of $\Delta_{t}$ 
with the first non-zero eigenvalue $\lambda(t):=\lambda_{1}(g_{t})$ of $g_{t}$ 
such that $\|u_{t}\|_{L^{2}(M)}=1$. In particular, $u_{t} \in \mathbf{1}^{\perp}$.  
Since $\Delta_{t}=t^{-2}\Delta_{V} +\Delta_{H}=t^{-2}\Delta_{M} +(1-t^{-2})\Delta_{H}$ 
(5.3 Proposition in \cite{BBB}) 
and $\Delta_{V}$ is non-negative, we see 
\[
\lambda(t) = \ispa{\Delta_{t} u_{t}, u_{t}}_{L^{2}(M)} 
\geq \ispa{\Delta_{H}u_{t},u_{t}}_{L^{2}(M)} \geq \msc{sg}. 
\]
Hence $\dsp \liminf_{t \to \infty} \lambda(t) \geq \msc{sg}$. 
The mini-max theorem shows that
\[
\lambda(t)=\inf_{0 \neq u \in \mathbf{1}^{\perp} \cap C^{\infty}(M)} 
\frac{\ispa{\Delta_{t}u,u}_{L^{2}(M)}}{\|u\|_{L^{2}(M)}^{2}}. 
\]
Take $u \in \mathbf{1}^{\perp} \cap C^{\infty}(M)$. Then, 
\[
\lambda(t) \|u\|_{L^{2}(M)}^{2} \leq \ispa{\Delta_{t} u,u}_{L^{2}(M)} =t^{-2}\ispa{\Delta_{M} u, u}_{L^{2}(M)} 
+(1-t^{-2}) \ispa{\Delta_{H}u,u}_{L^{2}(M)}. 
\]
Thus we have 
$\dsp \limsup_{t \to \infty} \lambda(t) \leq \frac{\ispa{\Delta_{H}u,u}_{L^{2}(M)}}{\|u\|_{L^{2}(M)}^{2}}$. 
Taking the infimum in $u \in \mathbf{1}^{\perp} \cap C^{\infty}(M)$, $u \neq 0$, 
we get $\dsp \limsup_{t \to \infty} \lambda(t) \leq \msc{sg}$. 
Thus we obtain 
\[
\msc{sg} \leq \liminf_{t \to \infty} \lambda_{1}(t) \leq \limsup_{t \to \infty} \lambda_{1}(t) \leq \msc{sg}. 
\]
Therefore, $\dsp \lim_{t \to \infty}\lambda_{1}(t)$ exists and equals $\msc{sg}$. 
\end{proof}
\subsection{Sunada's estimate and proof of Theorem \ref{BBC}}
According to Lemma $\ref{SG}$, to prove Theorem $\ref{BBC}$ 
it is enough to show that the spectral gap $\msc{sg}$ of $\Delta_{H}$ 
is positive under the assumption in Theorem $\ref{BBC}$. 
The representation $\rho \colon \mr{Hol}(b_{o}) \to U(L^{2}(F))$ is decomposed as 
$\rho={\bf 1} \oplus \rho_{o}$, where ${\bf 1}={\bf 1}_{\mr{Hol}(b_{o})}$ 
is the one-dimensional trivial sub-representation of $\rho$ on the space of constant functions $E(0)=\mb{C} \mbf{1}$, 
and $\rho_{o}$ is the orthogonal complement of ${\bf 1}$. 
In the statement of Theorem $\ref{BBC}$, $\rho$ is assumed to be ergodic. 
Thus $\rho_{o}$ has no non-zero invariant vector. 
Both of the representations ${\bf 1}$, $\rho_{o}$ define 
flat vector bundles $\mc{E}_{{\bf 1}}$, $\mc{E}_{\rho_{o}}$ 
and the twisted Laplacians, $\Delta_{{\bf 1}}$, $\Delta_{\rho_{o}}$, respectively. 
These operators are restrictions of $\Delta_{\rho}$ 
on $C^{\infty}(\mc{E}_{{\bf 1}})$ and on $C^{\infty}(\mc{E}_{\rho_{o}})$, respectively. 
Note that $\Delta_{{\bf 1}}$ is unitarily equivalent to the Laplacian $\Delta_{B}$ on $B$. 
Since $B$ is assumed to be connected, zero is a simple eigenvalue of $\Delta_{{\bf 1}}$. 
As our aim is to bound the bottom of the spectrum of $\Delta_{\rho}$ away from zero, 
one of important quantities is the first non-zero eigenvalue $\lambda_{1}(B)$ of $\Delta_{B}$. 
Likewise, for the bundle $\mc{E}_{\rho_{o}}$, 
an important quantity is the bottom of the spectrum of $\Delta_{\rho_{o}}$, defined by 
\begin{equation}\label{BTM}
\lambda_{0}(\rho_{o}) = 
\inf_{0 \neq s \in C^{\infty}(\mc{E}_{\rho_{o}})} 
\frac{
 \ispa{\Delta_{\rho_{o}}s, s}_{L^{2}(\mc{E}_{\rho_{o}})}
 }
 {
 \|s\|_{L^{2}(\mc{E}_{\rho_{o}})}^{2}
 }.
\end{equation}
The assumption in the statement of Theorem $\ref{BBC}$ that $\mr{Hol}(b_{o})$ satisfies 
{\it Kazhdan's property} (T) is to bound $\lambda_{0}(\rho_{o})$ from below. 
Let us review the definition of Kazhdan's property (T) following \cite{BHV}. 
Let $\Gamma$ be a finitely generated discrete group. 
For any finite generating set $Q$ in $\Gamma$ and 
any unitary representation $(H,\pi)$ of $\Gamma$, define 
\begin{equation}\label{PT1}
\kappa(\Gamma,Q,\pi) = \inf_{\xi \in H,\, \|\xi\|_{H}=1} \max_{x \in Q} \|\pi(x)\xi -\xi\|_{H}, 
\end{equation}
Define a quantity $\kappa(\Gamma,Q)$ by 
\begin{equation}\label{PT2}
\kappa(\Gamma,Q) = \inf_{\pi} \kappa(\Gamma, Q, \pi), 
\end{equation}
where $(H,\pi)$ runs over all equivalence classes of unitary representations of $\Gamma$ 
without non-zero invariant vectors. Then $\Gamma$ is said to have Kazhdan's Property (T)
if there exists a finite generating set $Q$ such that $\kappa(\Gamma, Q)>0$. 
Now, Sunada's inequality \cite{Su2} asserts that there exists a positive constant $C$ 
depending only on $B$ and $Q$ such that 
\begin{equation}\label{suI}
C \kappa(\mr{Hol}(b_{o}), Q, \pi)^{2} 
\leq \lambda_{0}(\pi), 
\end{equation}
where $\pi$ is any unitary representation of $\mr{Hol}(b_{o})$ 
on a separable Hilbert space, and $\lambda_{0}(\pi)$ is the bottom of the spectrum of $\Delta_{\pi}$. 
Therefore, applying $\eqref{suI}$ for $\pi=\rho_{o}$ and 
using the assumption that $\Gamma=\mr{Hol}(b_{o})$ has Kazhdan's Property (T), 
we obtain
\[
0<C \kappa(\mr{Hol}(b_{o}), Q)^{2}
\leq \lambda_{0}(\rho_{o}). 
\]
Therefore, to prove Theorem $\ref{BBC}$, it is enough to show the following. 
\begin{lem}
We have 
\[
\msc{sg}=\min\{\lambda_{1}(B), \lambda_{0}(\rho_{o})\}. 
\]
\end{lem}
\begin{proof}
Let $g_{0} \in C^{\infty}(B)$ be an $L^{2}$-normalized eigenfunction of $\Delta_{B}$ 
with the eigenvalue $\lambda_{1}(B)$. Set $\dsp f_{0}=\frac{1}{\mr{vol}(F)^{1/2}} \pi^{*}g_{o}$. 
Then $f_{0} \in \mbf{1}^{\perp} \cap C^{\infty}(M)$ and 
satisfies $\|f_{0}\|_{L^{2}(M)}=\|g_{0}\|_{L^{2}(B)}=1$. 
Since $\Delta_{H} \pi^{*}g_{0} =\pi^{*} \Delta_{B}g_{0}$, we see 
\[
\begin{split}
\lambda_{1}(B) & = \ispa{\Delta_{B}g_{0},g_{0}}_{L^{2}(B)} 
 = \frac{1}{\mr{vol}(F)} \ispa{\pi^{*} \Delta_{B} g_{o}, \pi^{*} g_{o}}_{L^{2}(M)} \\
& =\frac{1}{\mr{vol}(F)} \ispa{\Delta_{H}\pi^{*}g_{0}, \pi^{*}g_{0}}_{L^{2}(M)} 
=\ispa{\Delta_{H} f_{0}, f_{0}}_{L^{2}(M)} 
 \geq \msc{sg}. 
\end{split}
\]
The image of subspace $\mathbf{1}^{\perp} \cap C^{\infty}(M)$ by $W$ is  
$(W \mathbf{1})^{\perp} \cap C^{\infty}_{\mr{s}}(\mc{E}_{\rho})$. 
Since $C^{\infty}(\mc{E}_{\rho_{o}})$ is orthogonal to $C^{\infty}(\mc{E}_{{\bf 1}})$, 
and since $W{\bf 1} \in C^{\infty}(\mc{E}_{{\bf 1}})$, 
$C^{\infty}_{\mr{s}} (\mc{E}_{\rho_{o}})$ is contained in 
$(W \mathbf{1})^{\perp} \cap C^{\infty}_{\mr{s}}(\mc{E}_{\rho})$. 
Take $s \in C^{\infty}(\mc{E}_{\rho_{o}})$. Then, the same argument as in the proof of 
Lemma $\ref{UEL3}$ shows that there exists a sequence 
$s_{N} \in C^{\infty}_{\mr{s}}(\mc{E}_{\rho_{o}})$ such that $\|s-s_{N}\|_{L^{2}(\mc{E}_{\rho_{o}})} \to 0$ and 
$\|\Delta_{\rho_{o}}s - \Delta_{\rho_{o}}s_{N}\|_{L^{2}(\mc{E}_{\rho_{o}})} \to 0$ as $N \to \infty$. 
Hence, in the definition $\eqref{BTM}$ of $\lambda_{0}(\rho_{o})$, we can replace $C^{\infty}(\mc{E}_{\rho_{o}})$ 
by $C^{\infty}_{\mr{s}}(\mc{E}_{\rho_{o}})$ and obtain 
\[
\begin{split}
\msc{sg} & =\inf_{0 \neq s \in C^{\infty}_{\mr{s}}(\mc{E}_{\rho}) \cap (W \mathbf{1})^{\perp}}
\frac{\ispa{\Delta_{\rho}s,s}_{L^{2}(\mc{E}_{\rho})}}
{\|s\|_{L^{2}(\mc{E}_{\rho})}^{2}} 
 \leq \inf_{0 \neq s \in C^{\infty}_{\mr{s}}(\mc{E}_{\rho_{o}})} 
\frac{\ispa{\Delta_{\rho_{o}} s, s}_{L^{2}(\mc{E}_{\rho_{o}})} }{\|s\|_{L^{2}(\mc{E}_{\rho_{o}})}^{2}} 
 =\lambda_{0}(\rho_{o}). 
\end{split}
\]
This shows $\msc{sg} \leq \min \{\lambda_{1}(B), \lambda_{0}(\rho_{o})\}$. 
Next, take an arbitrary function $f$ in $\mathbf{1}^{\perp} \cap C^{\infty}(M)$ with $\|f\|_{L^{2}(M)}=1$ and 
decompose it as $f=Pf+f_{o}$, where $P$ is the orthogonal projection 
onto the closure of $\pi^{*}C^{\infty}(B)=C^{\infty}(M,0)$ in $L^{2}(M)$. 
The projection $P$ is concretely written as 
$P=\pi^{*} \comp A$, where $A \colon L^{2}(M) \to L^{2}(B)$ is a bounded operator given by 
\[
Af(b)=\frac{1}{\mr{vol}(F)} \int_{F_{b}} f\,d\mu_{F_{b}},\quad b \in B. 
\]
Set $\phi=Wf$ and $\phi_{0}=Wf_{0}$ so that $\phi = WPf +\phi_{0}$. 
We know $\phi_{0} \in C^{\infty}_{\mr{s}}(\mc{E}_{\rho_{o}})$
and $W Pf \in C^{\infty}(\mc{E}_{{\bf 1}})$. 
Since $\Delta_{\rho}$ preserves $C^{\infty}(\mc{E}_{{\bf 1}})$ and $C^{\infty}(\mc{E}_{\rho_{o}})$, 
$\Delta_{\rho}\phi_{0}=\Delta_{\rho_{o}}\phi_{0}$ is orthogonal to $WPf$. 
First, assume that $\Delta_{\mbf{1}} WPf=0$. Then $\Delta_{H} Pf=0$ and hence $Af$ is constant on $B$. 
But since $f \in \mbf{1}^{\perp}$, we have $Af=0$. This shows $Pf=0$. Hence $f=f_{o} \in C^{\infty}(M,0)^{\perp}$ 
and $\phi=\phi_{o} \in C^{\infty}(\mc{E}_{\rho_{o}})$. 
Thus, 
\[
\ispa{\Delta_{\rho}\phi, \phi}_{L^{2}(\mc{E}_{\rho})} = 
\ispa{\Delta_{\rho_{o}}\phi_{0},\phi_{0}}_{L^{2}(\mc{E}_{o})} 
\geq \|\phi_{0}\|_{L^{2}(\mc{E}_{\rho_{o}})}^{2} 
\lambda_{0}(\rho_{o}) 
\geq \|\phi\|_{L^{2}(\mc{E}_{\rho})}^{2}  \min \{\lambda_{1}(B), \lambda_{0}(\rho_{o})\}. 
\]
When $\Delta_{\mbf{1}} WPf \neq 0$, we see 
\[
\begin{split}
\ispa{\Delta_{\rho}\phi, \phi}_{L^{2}(\mc{E}_{\rho})} 
& =\ispa{\Delta_{\mbf{1}}WPf, WPf}_{L^{2}(\mc{E}_{\mbf{1}})} + 
\ispa{\Delta_{\rho_{o}}\phi_{0},\phi_{0}}_{L^{2}(\mc{E}_{o})} \\
& \geq \|WPf\|_{L^{2}(\mc{E}_{\mbf{1}})}^{2} \lambda_{1}(B) 
+\|\phi_{0}\|_{L^{2}(\mc{E}_{\rho_{o}})}^{2} 
\lambda_{0}(\rho_{o}) \\
& \geq ( \|WPf\|_{L^{2}(\mc{E}_{\mbf{1}})}^{2} + \|\phi_{0}\|_{L^{2}(\mc{E}_{\rho_{o}})}^{2} ) 
\min \{\lambda_{1}(B), \lambda_{0}(\rho_{o})\} \\
& = \|\phi\|_{L^{2}(\mc{E}_{\rho})}^{2} \min \{\lambda_{1}(B), \lambda_{0}(\rho_{o})\}. 
\end{split}
\]
Since this holds for any $\phi \in (W\mathbf{1})^{\perp} \cap C^{\infty}_{\mr{s}} (\mc{E}_{\rho})$, 
we have $\msc{sg} \geq \min \{\lambda_{1}(B), \lambda_{0}(\rho_{o})\}$. 
This concludes the assertion. 
\end{proof}

\section{Spectral coincidence for $\Delta_{H}$} \label{TM2}
This final section is devoted to a proof of Theorem $\ref{BOT}$. 
To prove Theorem $\ref{BOT}$, we need to compare the Laplacian $\Delta_{B_{H}}$ on $B_{H}$ 
and the twisted Laplacian $\Delta_{\rho}$ acting on $C^{\infty}(\mc{E}_{\rho})$. 
In the following, the unit element in $\mr{Hol}(b_{o})$ is denoted by $e$. 
\subsection{Periodification}
For any $\psi \in C_{0}^{\infty}(B_{H})$ and $f \in L^{2}(F)$, define 
the map $T(\psi,f) \colon B_{H} \to L^{2}(F)$ by 
\[
T(\psi,f)(x)=\sum_{h \in \mr{Hol}(b_{o})} (h^{-1}\psi)(x) \rho(h)^{-1}f \quad (x \in B_{H}), 
\]
where, for $\psi \in C^{\infty}(B_{H})$, $h \in \mr{Hol}(b_{o})$ and $x \in B_{H}$, we set $(h^{-1}\psi)(x)=\psi(hx)$. 
Since $\mr{Hol}(b_{o})$ acts on $B_{H}$ properly discontinuously and since $\supp \psi$ is compact, 
the sum in the right-hand side is locally finite. 
Since $T(\psi,f)$ is $\rho$-equivariant, 
$T(\psi,f)$ defines a smooth section of $\mc{E}_{\rho}$, which we tentatively write $s(\psi,f)$. 
The $\rho$-equivariant map from $B_{H}$ to $L^{2}(F)$ corresponding to the section $\Delta_{\rho}s(\psi,f)$ 
is $\Delta_{B_{H}} T(\psi,f)$. Since the sum defining $T(\psi,f)$ is locally finite, we see that 
$\Delta_{B_{H}}T(\psi,f)=T(\Delta_{B_{H}}\psi, f)$. Thus, identifying $s(\psi,f)$ and $T(\psi,f)$, we can write as 
\begin{equation}\label{DBH}
\Delta_{\rho}T(\psi,f)=T(\Delta_{B_{H}}\psi, f). 
\end{equation}
\begin{lem}\label{IP1}
For any $\psi_{1},\psi_{2} \in C_{0}^{\infty}(B_{H})$ and $f_{1},f_{2} \in L^{2}(F)$, we have 
\begin{equation}\label{PS1}
\ispa{T(\psi_{1},f_{1}), T(\psi_{2},f_{2})}_{L^{2}(\mc{E}_{\rho})} =
\sum_{h \in \mr{Hol}(b_{o})} \ispa{\rho(h)f_{1},f_{2}}_{L^{2}(F)} 
\ispa{h\psi_{1}, \psi_{2}}_{L^{2}(B_{H})}. 
\end{equation}
\end{lem}
\begin{proof}
Let $D \subset B_{H}$ be a fundamental domain of $\varpi_{H} \colon B_{H} \to B$. 
By definition, we have 
\[
\ispa{T(\psi_{1},f_{1}), T(\psi_{2},f_{2})}_{L^{2}(\mc{E}_{\rho})} =
\int_{D} \ispa{T(\psi_{1},f_{1})(x), T(\psi_{2},f_{2})(x)}_{L^{2}(F)} \,d\mu_{H}(x)
\]
The integrand of the integral in the right-hand side is computed as
\[
\ispa{T(\psi_{1},f_{1})(x), T(\psi_{2},f_{2})(x)}_{L^{2}(F)}=
\sum_{h,h' \in \mr{Hol}(b_{o})} \psi_{1}(hx) \ol{\psi_{2}(h'x)} \ispa{\rho(h)^{-1}f_{1}, \rho(h')^{-1}f_{2}}_{L^{2}(F)}. 
\]
Using the fact that the covering transformation group $\mr{Hol}(b_{o})$ acts isometrically on $B_{H}$, 
we compute the inner product $\ispa{T(\psi_{1},f_{1}), T(\psi_{2},f_{2})}_{L^{2}(\mc{E}_{\rho})}$ further as
\[
\begin{split}
\ispa{T(\psi_{1},f_{1}), T(\psi_{2},f_{2})}_{L^{2}(\mc{E}_{\rho})} 
& = \sum_{h,h'} \ispa{\rho(h'h^{-1})f_{1},f_{2}}_{L^{2}(F)} \int_{D} \psi_{1}(hx) \ol{\psi_{2}(h'x)} \,d\mu_{B_{H}}(x) \\
& = \sum_{h,h'} \ispa{\rho(h'h^{-1})f_{1},f_{2}}_{L^{2}(F)} \int_{h'D} \psi_{1}(hh'^{-1}y) \ol{\psi_{2}(y)} \,d\mu_{B_{H}}(y) \\
& = \sum_{g,h'} \ispa{\rho(g)f_{1},f_{2}}_{L^{2}(F)} \int_{h'D} \psi_{1}(g^{-1}y) \ol{\psi_{2}(y)} \,d\mu_{B_{H}}(y) \\
& = \sum_{g} \ispa{\rho(g)f_{1},f_{2}}_{L^{2}(F)}  \int_{\bigcup_{h' \in \mr{Hol}(b_{o})} h'D} 
\psi_{1}(g^{-1}y) \ol{\psi_{2}(y)} \,d\mu_{B_{H}}(y) \\
& = \sum_{g} \ispa{\rho(g)f_{1},f_{2}}_{L^{2}(F)}  \int_{B_{H}} 
\psi_{1}(g^{-1}y) \ol{\psi_{2}(y)} \,d\mu_{B_{H}}(y), \\
\end{split}
\]
which shows the lemma. 
\end{proof}
The next lemma will be used later to handle the terms 
$\ispa{\rho(g)f,f}_{L^{2}(F)}$ with $g \neq e$ in the sum $\eqref{PS1}$ 
to compute the norm-square $\|T(\psi,f)\|_{L^{2}(\mc{E}_{\rho})}^{2}$. 
\begin{lem}\label{IP2}
For any finite set $Q \subset \mr{Isom}(F)$, with $|Q| \geq 1$ and $e \not\in Q$, 
there exists a non-zero function $f_{Q} \in C^{\infty}(F)$ such that 
\begin{equation}\label{OT1}
\ispa{\rho(g)f_{Q},f_{Q}}_{L^{2}(F)}=0 \quad (g \in Q). 
\end{equation}
\end{lem}
\begin{proof}
Let $g \in \mr{Isom}(F)$, $g \neq e$. Then the set of fixed points 
\[
\mr{Fix}(g)=\{p \in F \mid g(p)=p\}
\]
is a finite union of closed submanifolds (\cite{KB}). 
Then, for any finite set $Q \subset \mr{Isom}(F)$ not containing $e$, 
the subset 
\[
G=F \setminus \bigcup_{g \in Q} \mr{Fix}(g)
\]
is a non-empty open set. Take $x_{o} \in G$. 
Let $g \in Q$. Since $gx_{o} \neq x_{o}$, 
we can take a neighborhood $U_{g}$ of $gx_{o}$ and a neighborhood $V_{g}$ of $x_{o}$ 
such that $\ol{U_{g}} \cap \ol{V_{g}}=\emptyset$. Since $x_{o} \in g^{-1}U_{g}$, the intersection 
$\dsp U=\bigcap_{h \in Q} (h^{-1}U_{h} \cap V_{h})$ 
is an open set containing $x_{o}$ satisfying $g\ol{U} \cap \ol{U}=\emptyset$ for any $g \in Q$. 
We take $f_{Q} \in C^{\infty}(F) \setminus \{0\}$ whose support is contained in $U$. 
If $g \in Q$, the support of $\rho(g)f_{Q}$ is contained in $gU$, and thus 
$\supp(\rho(g)f_{Q}) \cap \supp(f_{Q})=\emptyset$. Therefore, $\eqref{OT1}$ holds for this function $f_{Q}$. 
\end{proof}
The following is a first half of the assertions in Theorem $\ref{BOT}$. 
\begin{prop}\label{CCP}
We have $\spec(\Delta_{B_{H}}) \subset \spec(\Delta_{H})$. 
\end{prop}
\begin{proof}
Let $\lambda \in \spec(\Delta_{B_{H}})$. Take a sequence $\{\varphi_{k}\} \subset C_{0}^{\infty}(B_{H})$ 
such that $\|\varphi_{k}\|_{L^{2}(B_{H})}=1$ and 
$\|(\Delta_{B_{H}} -\lambda)\varphi_{k}\|_{L^{2}(B_{H})} \to 0$ as $k \to \infty$. 
Let $P_{k}$ be the subset in $\mr{Hol}(b_{o})$ defined by 
\[
P_{k}=\{g \in \mr{Hol}(b_{o}) \mid g\,\supp \varphi_{k} \cap \supp \varphi_{k} \neq \emptyset\}. 
\]
Since $\supp \varphi_{k}$ is compact, $P_{k}$ is a finite subset of $\mr{Hol}(b_{o})$. 
We take a function $f_{k}=f_{P_{k}} \in C^{\infty}(F)$ as in Lemma $\ref{IP2}$, and set 
$\phi_{k}=T(\varphi_{k},f_{k}) \in C^{\infty}(\mc{E}_{\rho})$. 
By Lemma $\ref{IP1}$, we have 
\[
\|\phi_{k}\|_{L^{2}(\mc{E}_{\rho})}^{2} = 
\sum_{h \in \mr{Hol}(b_{o})} \ispa{\rho(h)f_{k}, f_{k}} \ispa{h \varphi_{k}, \varphi_{k}}_{L^{2}(B_{H})}. 
\]
Let $h \in \mr{Hol}(b_{o})$. 
If a point $x \in B_{H}$ satisfies $\varphi_{k}(h^{-1}x) \neq 0$ and $\varphi_{k}(x) \neq 0$, 
then $x$ must be contained in $h \,\supp(\varphi_{k}) \cap \supp (\varphi_{k})$ and hence $h \in P_{k}$. 
Hence, 
\[
\begin{split}
\|\phi_{k}\|_{L^{2}(\mc{E}_{\rho})}^{2} 
& = \sum_{h \in P_{k}} \ispa{\rho(h)f_{k}, f_{k}}_{L^{2}(F)} 
\int_{B_{H}} \varphi_{k}(h^{-1}x) \ol{\varphi_{k}(x)} \,d\mu_{H}(x) \\
& = \|f_{k}\|^{2}_{L^{2}(F)} \|\varphi_{k}\|_{L^{2}(B_{H})}^{2} =\|f_{k}\|_{L^{2}(F)}^{2}. 
\end{split}
\]
By $\eqref{DBH}$, we have $(\Delta_{\rho}-\lambda)\phi_{k}=T((\Delta_{B_{H}} -\lambda)\varphi_{k},f_{k})$. 
Since the support of $(\Delta_{B_{H}} -\lambda)\varphi_{k}$ is contained in that of $\varphi_{k}$, 
the same computation as above works for $\|(\Delta_{\rho} -\lambda) \phi_{k}\|_{L^{2}(\mc{E}_{\rho})}^{2} $. 
Hence, 
\[
\|(\Delta_{\rho} -\lambda) \phi_{k}\|_{L^{2}(\mc{E}_{\rho})}^{2} 
=  \|(\Delta_{B_{H}} -\lambda)\varphi_{k}\|_{L^{2}(B_{H})}^{2} \|f_{k}\|_{L^{2}(F)}^{2}. 
\]
Therefore, we obtain 
\[
\frac{\|(\Delta_{\rho} -\lambda) \phi_{k}\|_{L^{2}(\mc{E}_{\rho})}^{2}}
{\|\phi_{k}\|_{L^{2}(\mc{E}_{\rho})}^{2}}
=\|(\Delta_{B_{H}} -\lambda)\varphi_{k}\|_{L^{2}(B_{H})}^{2} \to 0 \ \ (k \to \infty), 
\]
which concludes $\lambda \in \spec(\Delta_{\rho})=\spec(\Delta_{H})$. 
\end{proof}
\subsection{Construction of a Weyl sequence}
The main assertion in Theorem $\ref{BOT}$ is the following. 
\begin{prop}\label{ESSP}
Suppose that $\mr{Hol}(b_{o})$ is infinite. 
Then we have $\spec_{\mr{ess}}(\Delta_{B_{H}}) \subset \spec_{\mr{ess}}(\Delta_{\rho})$. 
\end{prop}
Take $\lambda \in \spec_{\mr{ess}}(\Delta_{B_{H}})$. 
Then we can take a sequence $\dsp \{\psi_{k}\}_{k=1}^{\infty}$ in $C_{0}^{\infty}(B_{H})$ satisfying 
\[
\ispa{\psi_{k}, \psi_{l}}_{L^{2}(B_{H})} =\delta_{k,l},\quad 
\|(\Delta_{B_{H}} -\lambda)\psi_{k}\|_{L^{2}(B_{H})} \to 0 \ \ (k \to \infty). 
\]
Thus, in what follows, we fix such a sequence $\{\psi_{k}\}$ in $C_{0}^{\infty}(B_{H})$. 
We want to construct a sequence of $L^{2}$-functions $\{f_{k}\}$ on $F$ such that 
the sequence of smooth sections $\{T(\psi_{k},f_{k})\}$ of $\mc{E}_{\rho}$ is a Weyl sequence, 
namely, it satisfies 
\begin{equation}\label{WL}
\ispa{T(\psi_{k}, f_{k}), T(\psi_{l},f_{l})}_{L^{2}(\mc{E}_{\rho})} =\delta_{k,l},\quad 
\|(\Delta_{\rho}-\lambda) T(\psi_{k},f_{k})\|_{L^{2}(\mc{E}_{\rho})} \to 0 \ \ (k \to \infty). 
\end{equation}
\par
\vspace{5pt}
\begin{rem}
Let us give a remark before describing a construction of such a $\{f_{k}\} \subset L^{2}(F)$. 
Namely, it could work to take certain sequence of eigenfunctions of $\Delta_{F}$ as $\{f_{k}\}$. 
To explain this, let $\mu,\nu$ be eigenvalues of $\Delta_{F}$ with $\mu \neq \nu$. 
Take $\varphi,\psi \in C_{0}^{\infty}(B_{H})$, $f \in E(\mu)$ and $g \in E(\nu)$. 
Then, since $E(\mu)$ is an invariant subspace of the representation $\rho$ of $\mr{Hol}(b_{o})$, 
$\ispa{\rho(h)f,g}_{L^{2}(F)}=0$ for any $h \in \mr{Hol}(b_{o})$. 
Thus according to Lemma $\ref{IP1}$, we have 
\[
\ispa{T(\varphi,f), T(\psi,g)}_{L^{2}(\mc{E}_{\rho})}=0.
\] 
Hence, suitably chosen sequence of eigenfunctions $\{f_{k}\}$ could give a 
Weyl sequence $\{T(\psi_{k},f_{k})\}$ for $\Delta_{\rho}$ if it satisfies $T(\psi_{k},f_{k}) \neq 0$ 
for infinitely many $k$. The equation $T(\psi_{k},f_{k})=0$ is equivalent to 
\[
\sum_{h \in \mr{Hol}(b_{o})} \psi_{k}(hx) \rho(h)^{-1}f =0
\]
in $L^{2}(F)$ for all $x \in B_{H}$. 
Hence it could be avoided by modifying the function $\psi_{k}$ keeping the property 
that $\{\psi_{k}\}$ is a Weyl sequence of $\Delta_{B_{H}}$. 
However it would need further discussions, and thus, 
instead of taking eigenfunctions as $\{f_{k}\}$, we construct, in the following, a suitable $\{f_{k}\}$ 
by a variant of the method used in the proof of Lemma $\ref{IP2}$. \hfill$\blacksquare$
\end{rem}
\par
\vspace{10pt}
For positive integers $k,l$, we set 
\[
P_{k,l}=\{g \in \mr{Hol}(b_{o}) \mid g \,\supp \psi_{k} \cap \supp \psi_{l} \neq \emptyset\}. 
\]
$P_{k,l}$ is a finite set, and we have $P_{l,k}=P_{k,l}^{-1}$, 
where, for any subset $P$ in $\mr{Hol}(b_{o})$, we write $P^{-1}=\{g^{-1} \mid g \in P\}$. 
For each positive integer $k$, define 
\[
P(k)=\bigcup_{l=1}^{k} \left(P_{k,l} \cup P_{k,l}^{-1}\right). 
\]
Note that $e \in P(k)$, and $g \in P(k)$ if and only if $g^{-1} \in P(k)$. 
\begin{lem}\label{TNe}
Suppose that we have sequences of open geodesic balls $\{X_{k}\}_{k=1}^{\infty}$, $\{X_{k}^{\pm}\}_{k=1}^{\infty}$ 
with radius less than the injectivity radius of $F$ such that 
\[
X_{1}^{+}=X_{1}, \quad X_{1}^{-}=\emptyset, \quad g \ol{X_{1}} \cap \ol{X_{1}}=\emptyset \ \  (g \in P(1) \setminus \{e\}), 
\]
and that, for $k \geq 2$, 
\begin{enumerate}
\item $\ol{X_{k}^{+}} \cap \ol{X_{k}^{-}}=\emptyset$ and $0<\mr{vol}(X_{k}^{+})=\mr{vol}(X_{k}^{-});$
\item $\ol{X_{k}^{+}} \cup \ol{X_{k}^{-}} \subset X_{k} \subset \ol{X_{k}} \subset X_{k-1}^{+};$
\item For any $g \in P(k) \setminus \{e\}$, we have $g \ol{X}_{k} \cap \ol{X}_{k}=\emptyset ;$ 
\item For any $g \in P(k)$, we have 
\[
g \ol{X_{k}} \cap 
\bigcup_{j=1}^{k-1} \left(\partial X_{j}^{+} \cup \partial X_{j}^{-} \right)
 =\emptyset. 
\]
\end{enumerate}
We define
\[
f_{1}=\mr{vol}(X_{1})^{-1/2}\mbf{1}_{X_{1}},\quad 
f_{k}=\left(2\mr{vol}(X_{k}^{+})\right)^{-1/2} \left(\mbf{1}_{X_{k}^{+}} - \mbf{1}_{X_{k}^{-}} \right) 
\ \ (k \geq 2), 
\] 
where $\mbf{1}_{A}$ denotes the characteristic function of a subset $A$ in $F$. 
Then, the sequence of smooth sections 
$\{T(\psi_{k},f_{k})\}_{k=1}^{\infty}$ of $\mc{E}_{\rho}$ 
is an orthonormal system in $L^{2}(\mc{E}_{\rho})$. 
\end{lem}
\begin{proof}
Let $g \in P(1) \setminus \{e\}$. 
Since $g \ol{X_{1}} \cap \ol{X_{1}} =\emptyset$ we have 
$\ispa{\rho(g)f_{1}, f_{1}}_{L^{2}(F)}=0$ and by definition of $f_{1}$, we see $\|f_{1}\|_{L^{2}(F)}=1$. 
Hence by Lemma $\ref{IP1}$, we have 
$\|T(\psi_{1},f_{1})\|_{L^{2}(\mc{E}_{\rho})}=1$. 
We need to show that, for any $k \geq 2$, 
\begin{equation}\label{TNeo}
\ispa{T(\psi_{k}, f_{k}), T(\psi_{i}, f_{i})}_{L^{2}(\mc{E}_{\rho})} = \delta_{i,k} \quad 
(i=1,\ldots,k). 
\end{equation}
We first note that, by (1), $f_{k} \neq 0$ and the integral of $f_{k} \comp \varphi$ 
is zero for any $\varphi \in \mr{Isom}(F)$. 
Let us consider the case $k=2$. Let $g \in P(2) \setminus\{e\}$. 
Since $g\ol{X_{2}}$ is connected, the conditions (3), (4) in Lemma $\ref{TNe}$ show that 
$g\ol{X_{2}} \subset \ol{X_{1}}^{c}$ or $g \ol{X_{2}} \subset X_{1} \setminus \ol{X_{2}}$. 
If $g \ol{X_{2}} \subset \ol{X_{1}}^{c}$, then $\rho(g) f_{2} \cdot f_{1}=0$ because 
$\supp(\rho(g)f_{2})=g \, \supp(f_{2}) \subset g \ol{X_{2}}$ and $\supp (f_{1})=\ol{X_{1}}$.  
If $g \ol{X_{2}} \subset X_{1} \setminus \ol{X_{2}}$, we see $\rho(g) f_{2} \cdot f_{1}=\rho(g)f_{2}$. 
By (3), we have $\rho(g)f_{2} \cdot f_{2}=0$. 
Hence $\ispa{\rho(g)f_{2}, f_{2}}_{L^{2}(F)}=\ispa{\rho(g)f_{2}, f_{1}}_{L^{2}(F)}=0$. 
From this and Lemma $\ref{IP1}$, we have 
\[
\|T(\psi_{2},f_{2})\|_{L^{2}(\mc{E}_{\rho})}^{2} =1, \quad 
\ispa{T(\psi_{2},f_{2}), T(\psi_{1},f_{1})}_{L^{2}(\mc{E}_{\rho})} 
= \ispa{f_{2},f_{1}}_{L^{2}(F)} \ispa{\psi_{2},\psi_{1}}_{L^{2}(B_{H})}=0, 
\]
because $\{\psi_{k}\}_{k=1}^{\infty}$ is an orthonormal system in $L^{2}(B_{H})$. 
Next, we show $\eqref{TNeo}$ for $k \geq 3$. Take $g \in P(k) \setminus \{e\}$. 
By (3), we see $\rho(g) f_{k} \cdot f_{k} =0$, and thus, by Lemma $\ref{IP1}$, we have 
$\|T(\psi_{k},f_{k})\|_{L^{2}(\mc{E}_{\rho})}^{2}=1$. 
We note that, by (2) and $X_{1}^{-}=\emptyset$, the following holds; 
\[
\bigcap_{j=1}^{k-1} (\partial X_{j}^{+} \cup \partial X_{j}^{-})^{c}=
\ol{X_{1}}^{c} \cup X_{k-1}^{+} \cup 
\bigcup_{j=2}^{k-1} X_{j}^{-} \cup \bigcup_{j=2}^{k-1} 
\left[
X_{j-1}^{+} \setminus \left(\ol{X_{j}^{+}} \cup \ol{X_{j}^{-}}\right)
\right], 
\]
and the unions in the right-hand side are all disjoint. Take $g \in P(k) \setminus \{e\}$. 
Then, by (3) and the above formula, we have 
\[
g\ol{X_{k}} \subset \ol{X_{1}}^{c} \cup \left(X_{k-1}^{+} \setminus \ol{X_{k}} \right) \cup 
\bigcup_{j=2}^{k-1} X_{j}^{-} \cup \bigcup_{j=2}^{k-1} 
\left[
X_{j-1}^{+} \setminus \left(\ol{X_{j}^{+}} \cup \ol{X_{j}^{-}}\right)
\right]. 
\]
Since $g\ol{X_{k}}$ is connected, it is contained in one of connected components 
of the set in the right-hand side above. 
Let $i$ be an integer with $1 \leq i \leq k-1$. If $g \ol{X_{k}} \subset \ol{X_{1}}^{c}$, 
then $f_{i}=0$ on $\supp(\rho(g) f_{k})$. Hence $\rho(g)f_{k} \cdot f_{i}=0$ and 
\begin{equation}\label{ZER}
\ispa{\rho(g)f_{k},f_{i}}_{L^{2}(F)}=0. 
\end{equation}
Suppose $g \ol{X_{k}} \subset X_{k-1}^{+} \setminus \ol{X_{k}}$. 
Since $X_{k-1}^{+} \subset X_{i}^{+}$, we have $f_{i}=1$ on $g \ol{X_{k}}$, and hence $\rho(g)f_{k} \cdot f_{i}=\rho(g)f_{k}$ on $F$. 
Since the integral of $\rho(g)f_{k}$ is zero, we have $\eqref{ZER}$. 
Suppose $g \ol{X_{k}} \subset X_{j}^{-}$ for some $j$ with $2 \leq j \leq k-1$. 
If $i<j$, then $\ol{X_{j}^{-}} \subset X_{i}^{+}$, and hence $f_{i}=1$ on $g \ol{X_{k}}$. 
Thus $\rho(g)f_{k} \cdot f_{i}=\rho(g) f_{k}$, and hence we get $\eqref{ZER}$. 
If $i > j$, then $\ol{X_{i}^{-}}  \cup \ol{X_{i}^{+}} \subset X_{i-1}^{+} \subset X_{j}^{+}$. 
Since $X_{j}^{+} \cap X_{j}^{-}=\emptyset$ by the condition (1),  
it holds that $\supp (\rho(g)f_{k}) \cap \supp(f_{i})=\emptyset$. 
Hence $\rho(g) f_{k} \cdot f_{i}=0$ and we get $\eqref{ZER}$. 
If $i=j$, then $f_{i}=-1$ on $\supp(\rho(g)f_{k})$, and hence $\rho(g)f_{k} \cdot f_{i}=-\rho(g)f_{k}$ on $F$. 
Thus we have $\eqref{ZER}$. 
Finally, suppose that $g\ol{X_{k}} \subset X_{j-1}^{+} \setminus \left(\ol{X_{j}^{+}} \cup \ol{X_{j}^{-}} \right)$ 
for some $j$ with $2 \leq j \leq k-1$. If $i \geq j$, then 
\[
\supp(f_{i}) =\ol{X_{i}^{+}} \cup \ol{X_{i}^{-}} \subset \ol{X_{j}^{+}} \cup \ol{X_{j}^{-}}. 
\]
This shows $\rho(g)f_{k} \cdot f_{i}=0$ and hence $\eqref{ZER}$. If $i<j$, then $X_{j-1}^{+} \subset X_{i}^{+}$ 
and hence $f_{i}=1$ on $\supp(\rho(g) f_{k})$. Thus $\rho(g)f_{k} \cdot f_{i}=\rho(g)f_{k}$ and hence $\eqref{ZER}$. 
Therefore, we obtain $\eqref{ZER}$ for any $i$ with $1 \leq i \leq k-1$ and any $g \in P(k) \setminus \{e\}$. 
Again by Lemma $\ref{IP1}$, we have $\eqref{TNeo}$ for $i=1,\ldots,k-1$. This completes the proof. 
\end{proof}
\noindent{\bf Proof of Proposition \ref{ESSP}} 
First of all, we construct sequences of open geodesic balls satisfying the conditions in Lemma $\ref{TNe}$. 
By using the same method as in the proof of Lemma $\ref{IP2}$, 
we can take an open geodesic ball $X_{1}$ small enough so that 
the radius of $X_{1}$ is less than the injectivity radius of $F$ satisfying 
\[
g \ol{X_{1}} \cap \ol{X_{1}} =\emptyset \quad (g \in P(1) \setminus \{e\}). 
\]
We set $X_{1}^{+}=X_{1}$ and $X_{1}^{-}=\emptyset$. 
Next, we take 
\[
x_{1} \in G_{1}=X_{1}^{+} \setminus 
\left[
\bigcup_{g \in P(2)} g^{-1} (\partial X_{1}^{+} \cup \partial X_{1}^{-}) \cup \bigcup_{g \in P(2) \setminus \{e\}} \mr{Fix}(g)
\right]. 
\]
Note that the set in the right-hand side above is nonempty, 
for the subset subtracted from $X_{1}=X_{1}^{+}$ is a finite union of closed submanifolds 
and hence its complement is open and dense in $F$. 
The point $x_{1}$ satisfies, for any $g \in P(2)$, 
\[
g x_{1} \not\in \partial X_{1}, \quad g x_{1} \neq x_{1} \ \  (g \neq e). 
\]
For each $g \in P(2)$, we take connected open neighborhoods 
$U_{g}^{1}$ of $gx_{1}$ and $V_{g}^{1}$ of $x_{1}$, respectively, such that 
$x_{1} \in U_{e}^{1}=V_{e}^{1} \subset \ol{U_{e}^{1}} \subset G_{1}$, and that 
\[
\ol{U_{g}^{1}} \cap \partial X_{1} =\emptyset,\quad \ol{U_{g}^{1}} \cap \ol{V_{g}^{1}} =\emptyset \ \ (g \neq e). 
\]
Since $x_{1} \in g^{-1}U_{g}^{1} \cap V_{g}^{1}$ for any $g \in P(2)$, we can take 
an open geodesic ball $X_{2}$ centered at $x_{1}$ having the radius small enough as before satisfying 
\[
\ol{X_{2}} \subset \bigcap_{h \in P(2)} (h^{-1}U_{h}^{1} \cap V_{h}^{1}). 
\]
We fix such an $X_{2}$. 
Then $\ol{X_{2}} \subset X_{1}$, and, for all $g \in P(2)$, we have 
$g \ol{X_{2}} \cap \partial X_{1}=\emptyset$ and $g \ol{X_{2}} \cap \ol{X_{2}}=\emptyset$ ($g \neq e$). 
Take any open geodesic balls $X_{2}^{\pm} \subset X_{2}$ such that 
\[
\ol{X_{2}^{+}} \cup \ol{X_{2}^{-}} \subset X_{2},\quad 
\ol{X_{2}^{+}} \cap \ol{X_{2}^{-}} =\emptyset,\quad 
0<\mr{vol}(X_{2}^{+})=\mr{vol}(X_{2}^{-}). 
\]
In this way we get $\{X_{k}\}_{k=1}^{2}$, $\{X_{k}^{\pm}\}_{k=1}^{2}$, which satisfy 
the four conditions in Lemma $\ref{TNe}$. 
\par
Let $N$ be an integer with $N \geq 2$ and assume that 
we have constructed sequences $\{X_{k}\}_{k=1}^{N}$, $\{X_{k}^{\pm}\}_{k=1}^{N}$ 
satisfying the four conditions in Lemma $\ref{TNe}$. 
Then we take geodesic balls $X_{N+1}$, $X_{N+1}^{\pm}$ in the following way. 
First, take a point 
\begin{equation}\label{WSEN}
x_{N} \in G_{N}=X_{N}^{+} \setminus 
\left[
\bigcup_{g \in P(N+1)} 
\bigcup_{j=1}^{N} g^{-1} (\partial X_{j}^{+} \cup \partial X_{j}^{-}) \cup 
\bigcup_{g \in P(N+1) \setminus \{e\}} \mr{Fix}(g)
\right]. 
\end{equation}
Then, for any $g \in P(N+1)$, we have 
\[
gx_{N} \not\in \bigcup_{j=1}^{N} \left(\partial X_{j}^{+} \cup \partial X_{j}^{-}\right),\quad 
gx_{N} \neq x_{N} \ \ (g \neq e). 
\]
For $g \in P(N+1)$, we take open neighborhoods $U_{g}^{N}$ of $gx_{N}$ and $V_{g}^{N}$ of $x_{N}$, respectively, 
such that $x_{N} \in U_{e}^{N}=V_{e}^{N} \subset \ol{U_{e}^{N}} \subset G_{N}$, and that 
\[
\ol{U_{g}^{N}} \cap \bigcup_{j=1}^{N} \left( \partial X_{j}^{+} \cup \partial X_{j}^{-}\right)=\emptyset,\quad 
\ol{U_{g}^{N}} \cap \ol{V_{g}^{N}}=\emptyset \ \ (g \neq e). 
\]
Next, we take an open geodesic ball $X_{N+1}$ centered at $x_{N}$ such that 
\[
\ol{X_{N+1}} \subset \bigcap_{h \in P(N+1)} \left(h^{-1}U_{h}^{N} \cap V_{h}^{N}\right). 
\]
We have $g \ol{X_{N+1}} \subset U_{g}^{N}$ and $\ol{X_{N+1}} \subset V_{g}^{N}$ 
for any $g \in P(N+1)$, and thus the following holds; 
\[
g \ol{X_{N+1}} \cap \bigcup_{j=1}^{N} \left(\partial X_{j}^{+} \cup \partial X_{j}^{-}\right)=\emptyset,\quad 
g \ol{X_{N+1}} \cap \ol{X_{N+1}} =\emptyset \ \ (g \neq e). 
\]
Take arbitrary open geodesic balls $X_{N+1}^{\pm} \subset X_{N+1}$ such that 
\[
\ol{X_{N+1}^{+}} \cup \ol{X_{N+1}^{-}} \subset X_{N+1},\quad 
\ol{X_{N+1}^{+}} \cap \ol{X_{N+1}^{-}} =\emptyset,\quad 
0<\mr{vol}(X_{N+1}^{+})=\mr{vol}(X_{N+1}^{-}). 
\]
We inductively construct sequences $\{X_{k}\}_{k=1}^{\infty}$ and $\{X_{k}^{\pm}\}_{k=1}^{\infty}$, and 
these sequences clearly satisfy the four conditions in Lemma $\ref{TNe}$. 
Let $\{f_{k}\}_{k=1}^{\infty}$ be the sequence of $L^{2}$-functions 
on $F$ constructed in Lemma $\ref{TNe}$ so that the sequence of smooth sections 
$\{T(\psi_{k},f_{k})\}_{k=1}^{\infty}$ is an orthonormal system in $L^{2}(\mc{E}_{\rho})$. Now we have 
\[
\|(\Delta_{\rho}-\lambda) T(\psi_{k},f_{k}) \|_{L^{2}(\mc{E}_{\rho})}^{2} 
=\sum_{g \in \mr{Hol(b_{o})}} \ispa{\rho(g)f_{k},f_{k}}_{L^{2}(F)} 
\ispa{g(\Delta_{B_{H}}-\lambda)\psi_{k}, (\Delta_{B_{H}}-\lambda)\psi_{k}}_{L^{2}(B_{H})}.
\]
Since $\supp((\Delta_{B_{H}} -\lambda)\psi_{k}) \subset \supp(\psi_{k})$, 
it is enough to take the sum in the above over all $g \in P(k)$. By the construction of $f_{k}$, 
we have $\ispa{\rho(g)f_{k},f_{k}}_{L^{2}(F)}=0$ for $g \in P(k) \setminus \{e\}$. Hence 
\[
\|(\Delta_{\rho}-\lambda) T(\psi_{k},f_{k}) \|_{L^{2}(\mc{E}_{\rho})}^{2} 
=\|f_{k}\|_{L^{2}(F)}^{2} \|(\Delta_{B_{H}}-\lambda)\psi_{k}\|_{L^{2}(B_{H})}^{2}
=\|(\Delta_{B_{H}}-\lambda)\psi_{k}\|_{L^{2}(B_{H})}^{2} \to 0
\]
as $k \to \infty$. This shows $\lambda \in \spec_{\mr{ess}}(\Delta_{B_{H}})$. 
\hfill$\blacksquare$

\par
\vspace{10pt}
\noindent{\bf Proof of Theorem \ref{BOT}: }\hspace{2pt} 
We have proved the assertion that $\spec(\Delta_{B_{H}}) \subset \spec(\Delta_{H})$. 
Polymerakis shows in \cite{Pol} that the Laplacian on the Riemannian covering over a compact Riemannian 
manifold with infinite covering transformation group does not have an eigenvalue with finite multiplicity. 
Hence Proposition $\ref{ESSP}$ gives 
\[
\spec(\Delta_{B_{H}})=\spec_{\mr{ess}} (\Delta_{B_{H}}) \subset \spec_{\mr{ess}} (\Delta_{\rho}) =\spec_{\mr{ess}} (\Delta_{H}). 
\]
Suppose that $\mr{Hol}(b_{o})$ is amenable. Then by Sunada' theorem (\cite{Su1}, Corollary to Theorem 2) 
$\spec(\Delta_{H})=\spec(\Delta_{\rho}) \subset \spec(\Delta_{B_{H}})$. Hence we have 
the equality $\spec(\Delta_{H})=\spec(\Delta_{B_{H}})$. Furthermore, if $\mr{Hol}(b_{o})$ is infinite and amenable, 
we have 
\[
\spec_{\mr{ess}} (\Delta_{B_{H}})  \subset \spec_{\mr{ess}} (\Delta_{H}) 
\subset \spec(\Delta_{H}) \subset \spec(\Delta_{B_{H}}) = \spec_{\mr{ess}} (\Delta_{B_{H}}), 
\]
from which the equality holds. Therefore, we conclude Theorem $\ref{BOT}$. 
\hfill$\blacksquare$
\par
\vspace{10pt}
\noindent{\bf Proof of Corollary \ref{COR2}:} \hspace{2pt}
In \cite{ON2}, it is proved that, if $\gamma \colon I \to M$ is a geodesic in $M$ 
defined on an interval $I$ such that $\gamma'(t_{o}) \in \mc{H}_{\gamma(t_{o})}$  
at a point $t_{o} \in I$, then $\gamma'(t) \in \mc{H}_{\gamma(t)}$ for all $t \in I$. 
Under our assumption that the horizontal distribution is integrable, 
each leaf of $\mc{H}$ is totally geodesic. 
Let $p \in F$ and let $L$ be the leaf of $\mc{H}$ through $p$. 
Since $M$ is compact and hence complete, $L$ is also complete. 
Since the restriction $\pi_{L} \colon L \to B$ of the submersion 
$\pi \colon M \to B$ to $L$ is locally isomorphic, it turns out that it is a Riemannian covering. 
Let $\mr{Hol}(b_{o})_{p}$ be the stabilizer of $p$ in $\mr{Hol}(b_{o})$. 
Then we have 
\[
\mf{h}^{-1}(\mr{Hol}(b_{o})_{p}) =(\pi_{L})_{\#} \pi_{1}(L,p), 
\]
where $(\pi_{L})_{\#} \colon \pi_{1}(L,p) \to \pi_{1}(B,b_{o})$ is the homomorphism 
induced by $\pi_{L} \colon L \to B$. Therefore, if $\mr{Hol}(b_{o})_{p}=\{1\}$, 
then we have $(\pi_{L})_{\#} \pi_{1}(L,p) =\ker(\mf{h})$. 
From this it follows that $\pi_{L} \colon L \to B$ is isomorphic, as a Riemannian covering space, 
to $\varpi_{H} \colon B_{H} \to B$. Therefore, Corollary $\ref{COR2}$ is a direct 
consequence of Theorem $\ref{BOT}$. 
\hfill$\blacksquare$

\end{document}